\documentclass[11pt,twoside,a4paper]{article}

\usepackage[british]{babel}
\usepackage[T1]{fontenc}

\usepackage[utf8]{inputenc}

\usepackage{csquotes}
\usepackage[backend=biber, maxbibnames=10, maxcitenames=4, maxalphanames=4, bibencoding=utf8, style=alphabetic, isbn=false, url=false, doi=true, giveninits=true]{biblatex}
\renewbibmacro{in:}{}
\renewbibmacro*{doi+eprint+url}{%
  \printfield{doi}%
  \newunit\newblock%
  \iftoggle{bbx:eprint}{%
    \usebibmacro{eprint}%
  }{}%
  \newunit\newblock%
  \iffieldundef{doi}{%
  \usebibmacro{url+urldate}}%
  {}%
}
\AtEveryBibitem{\clearfield{month}}
\AtEveryBibitem{\clearfield{note}}
\usepackage{setspace}
\usepackage{lmodern}
\usepackage[indentafter]{titlesec}
\titleformat{name=\section}{}{\thetitle.}{0.8em}{\centering\scshape}
\titleformat{name=\subsection}[runin]{}{\thetitle.}{0.5em}{\bfseries}[.]
\titleformat{name=\subsubsection}[runin]{}{\thetitle.}{0.5em}{\itshape}[.]
\titleformat{name=\subparagraph,numberless}[runin]{}{}{0em}{}[.]
\titlespacing{\subparagraph}{0em}{0em}{0.5em}
\usepackage{url}
\usepackage{mathtools}
\DeclarePairedDelimiter\abs{\lvert}{\rvert}%
\DeclarePairedDelimiter\norm{\lVert}{\rVert}%

\makeatletter
\let\oldabs\abs
\def\abs{\@ifstar{\oldabs}{\oldabs*}}
\let\oldnorm\norm
\def\norm{\@ifstar{\oldnorm}{\oldnorm*}}
\makeatother

\usepackage{amssymb}
\usepackage{amsthm}
\usepackage{mathrsfs}
\usepackage{dsfont}
\usepackage{enumitem}
\setlist[enumerate]{noitemsep, partopsep=0pt, topsep=0pt, parsep=0pt, itemsep=0pt}
\setlist[itemize]{noitemsep, partopsep=0pt, topsep=0pt, parsep=0pt, itemsep=0pt}

\usepackage{interval}
\intervalconfig{soft open fences}
\usepackage[normalem]{ulem}
\usepackage[stretch=10]{microtype}
\usepackage[affil-it, noblocks]{authblk}

\usepackage{xcolor}

\usepackage[hidelinks,draft=false]{hyperref}
\hypersetup{
  colorlinks   = true, %Colours links instead of ugly boxes
  urlcolor     = blue, %Colour for external hyperlinks
  linkcolor    = blue, %Colour of internal links
  citecolor   = red %Colour of citations
}

\usepackage[nameinlink,noabbrev,capitalize]{cleveref}
\crefname{equation}{}{}
\usepackage[plain]{fullpage}
\newlist{theoenum}{enumerate}{1} % also creates a counter called 'propenumi'
\setlist[theoenum]{label=\normalfont(\roman*), ref=\theproposition~\normalfont(\roman*), noitemsep, partopsep=0pt, topsep=0pt, parsep=0pt, itemsep=0pt}
\crefalias{theoenumi}{theorem}

\usepackage{tocloft}
\usepackage{titlefoot}

\usepackage{todonotes}
 \newcommand{\blue}{\color{blue}}

\hypersetup{
    colorlinks,
    linkcolor={red!50!black},
    citecolor={blue!50!black},
    urlcolor={blue!80!black},
}

\DeclareMathOperator{\supp}{supp}

\newcommand*\diff{\mathop{}\!\mathrm{d}}

\newcommand{\N}{\mathbb{N}}
\newcommand{\R}{\mathbb{R}}

\newcommand{\Z}{\mathbb{Z}}
\newcommand{\p}{\mathtt p}
\newcommand{\bbH}{\mathbb{H}}

\newcommand{\prob}{\mathscr P}

\newcommand{\vol}{\mathrm{dvol}}
\newcommand\restr[2]{{% we make the whole thing an ordinary symbol
  \left.\kern-\nulldelimiterspace % automatically resize the bar with \right
  #1 % the function
  \right|_{#2} % this is the delimiter
  }}

\newcommand{\cd}{\mathsf{CD}}

\newcommand{\RCD}{\mathsf{RCD}}

\newcommand{\Ric}{\mathrm{Ric}}

\newcommand{\X}{\mathsf{X}}

\newcommand{\Lip}{\mathsf {Lip}}

\newcommand{\di}
{\mathsf d} %standard mms distance notation
\newcommand{\m}{\mathfrak m} %standard mms measure notation

\DeclareMathOperator{\Geo}{Geo}

\newcommand{\Prob}{\mathscr{P}}

\newcommand{\Sbb}{\mathbb{S}}

\renewcommand{\S}{\mathbb{S}}

\renewcommand{\epsilon}{\varepsilon}
\renewcommand{\phi}{\varphi}
\newcommand{\hes}{\mathrm{Hess}}
\newcommand{\G}{\mathbb{G}}
\newcommand{\Scal}{\mathcal{S}}
\date{\today}
\title{\textbf{\uppercase{\large{Curvature-dimension condition of sub-Riemannian $\alpha$-Grushin half-spaces}}}}
\author{Samu\"el Borza\footnote{University of Vienna, Universit\"atsring 1, 1010 Vienna, Austria. \textit{E-mail}:  \href{mailto:samuel.borza@univie.ac.at}{samuel.borza@univie.ac.at}}, \ 
and Kenshiro Tashiro\footnote{Okinawa Institute of Science and Technology, 1919-1 Tancha, Onna-son, Kunigami-gun, Okinawa, Japan. 
\textit{E-mail}:  \href{mailto:kenshiro.tashiro@oist.jp}{kenshiro.tashiro@oist.jp}}}

\newtheoremstyle{remark}% <name>
        {10pt}% <Space above>
        {10pt}% <Space below>
        {}% <Body font>
        {}% <Indent amount>
        {\itshape}% <Theorem head font>
        {.}% <Punctuation after theorem head>
        {.4em}% <Space after theorem head>
        {}% <Theorem head spec (can be left empty, meaning 'normal')>

\newtheoremstyle{proof}% <name>
        {10pt}% <Space above>
        {10pt}% <Space below>
        {}% <Body font>
        {}% <Indent amount>
        {\itshape}% <Theorem head font>
        {.}% <Punctuation after theorem head>
        {.4em}% <Space after theorem head>
        {}% <Theorem head spec (can be left empty, meaning 'normal')>

\newtheoremstyle{definition}% <name>
        {10pt}% <Space above>
        {10pt}% <Space below>
        {}% <Body font>
        {}% <Indent amount>
        {\bfseries}% <Theorem head font>
        {.}% <Punctuation after theorem head>
        {.4em}% <Space after theorem head>
        {}% <Theorem head spec (can be left empty, meaning 'normal')>

\newtheoremstyle{theorem}% <name>
        {10pt}% <Space above>
        {10pt}% <Space below>
        {\slshape}% <Body font>
        {}% <Indent amount>
        {\bfseries}% <Theorem head font>
        {.}% <Punctuation after theorem head>
        {.4em}% <Space after theorem head>
        {}% <Theorem head spec (can be left empty, meaning 'normal')>

\theoremstyle{theorem}
\newtheorem{theorem}{Theorem}[section]
\newtheorem{prop}[theorem]{Proposition}

\newtheorem{lemma}[theorem]{Lemma}

\theoremstyle{definition}
\newtheorem{definition}[theorem]{Definition}

\theoremstyle{remark}
\newtheorem{remark}[theorem]{Remark}

\theoremstyle{proof}
\newtheorem*{pro}{Proof}
 {\popQED\end{pro}}

\makeatletter
\renewcommand\xleftrightarrow[2][]{%
  \ext@arrow 9999{\longleftrightarrowfill@}{#1}{#2}}
\newcommand\longleftrightarrowfill@{%
  \arrowfill@\leftarrow\relbar\rightarrow}
\makeatother

\makeindex
\begin{document}

\maketitle

\providecommand{\keywords}[1]
{
	\textbf{\textit{Keywords---}} #1
}

\providecommand{\msc}[1]
{
	\textbf{\textit{MSC (2020)---}} #1
}

\begin{abstract}
	 We provide new examples of sub-Riemannian manifolds with boundary equipped with a smooth measure that satisfy the $\mathsf{RCD}(K , N)$ condition. They are constructed by equipping the half-plane, the hemisphere and the hyperbolic half-plane with a two-dimensional almost-Riemannian structure and a measure that vanishes on their boundary. The construction of these spaces is inspired from the geometry of the $\alpha$-Grushin plane.
\end{abstract}


\msc{53C17, 53C21, 49Q22}

% Sam: My idea of structure (suggestion, I am happy to change if you think that something else is better).\\
% \textbf{Section 1} introduction \\
% \textbf{Section 2} The $\alpha$-spaces and halspaces fully described as metric measure space (including the measure). One subsection for each model (three in total).\\
% \textbf{Section 3} Equivalence CD and Ric \\
% \textbf{Section 4} Computation of the Ricci curvatures\\
% What do you think?
% \\
% Ken: I like this structure. But what is the $\alpha$-space?
% In my mind there are three types:\\
% (i)Grushin spaces with $\alpha=(\alpha_1,\dots,\alpha_k)$,\\
%    (ii) Grushin spaces with $\alpha=(\alpha,0,\dots,0)$,\\
%     (iii) Grushin surfaces\\
% If we adopt (iii), then the tentative $\alpha$-Grushin model spaces would be nice.
% (ii), I think we can introduce $n$-dim. model space with a small effort.
% (i), It might be possible, but it is too technical for non-sub-Riemannian researchers. \\
% Sam: Sorry when I said space I just meant the three $\alpha$-models that we are looking at. I am not sure it's worth thinking too much about $n$-dimensional ones. Of course it's interesting and we could do it if you want.\\
% ken: Ok! I'm trying to calculate $n$-dimensional case. Let's decide after the computation is done\\
% Sam: sure! I keep going on
\tableofcontents
% \todo[inline]{Ken: I prefer $\alpha$-Grushin surfaces as a whole name. How do you think?\\ Sam : To me a surface is an $(n - 1)$-dimensional submanifold of an $n$-dim manifold \\ Sam: Note that $\mathsf{RCD}(K, 2)$ can only be satisfied if the measure is proportional to the Riemannian volume.\\
% Ken: Ok, then all ``$\alpha$-Grushin model spaces'' will be changed into ``$\alpha$-Grushin spaces''?}
\section{Introduction}
\label{section:Introduction}

% \todo[inline]{
%     1. Let Riemannian model space $g_K=dx^2 + f_K(x)^2dy^2$,
%     and Grushin metric is $g_\alpha = dx^2 + \frac{f_K(x)^2}{w(x)^{2\alpha}}dy^2$. Add the computation at the begin of section 5, introduce the model in the introduction and unify in section 3 every time we introduce the model.\\
%     33. Check the computation of $\infty$-Grushin (section 5)\\
%     9. Change notation tangent space. Sam: seems DONE by Kenshiro. Sam: I changed more mathrm Ken: I have checked, complete.\\ 12. Change $T$ to $T = 1$. I think also DONE.}
In the past few decades, the curvature-dimension condition $\mathsf{CD}(K, N)$ was specifically developed to generalise to the non-smooth setting of metric measure spaces
% {\red this bound}
the concept of a lower bound on the $N$-Bakry-Émery Ricci tensor $\mathrm{Ric}_{N,V} \geq K$ on weighted Riemannian manifolds, using the theory of optimal transport (see \cite{villani,lott--villani2009, sturm2006}).
The $\RCD(K,N)$ condition is a restriction of the $\cd(K,N)$ condition by requiring in addition that the space is infinitesimally Hilbertian, which excludes Finsler manifolds from being $\RCD$-spaces (see \cite{ambrosioICMs, Gigli12,Erbar2015} for example).
The conditions $\cd(K,N)$ and $\RCD(K,N)$ have not only established known meaningful geometric inequalities to non-smooth spaces, but also helped obtain new results, even in the smooth setting.

However, it has been found that the large class of sub-Riemannian manifolds equipped with a smooth positive measure does not satisfy any of the $\mathsf{CD}$ conditions. In the Heisenberg group, this was first proven in \cite{juillet2005} and then for any Carnot groups in \cite{ambrosio2020}. The same author then showed in \cite{juillet2020} that the same result holds for any sub-Riemannian manifold whose distribution has constant rank strictly smaller than its topological dimension. A no-$\mathsf{CD}$ theorem for two-dimensional rank-varying structures was then found in \cite{magnaboscorossi}. Finally, the most general result to date, established in \cite{rizzi2023failure}, states that no curvature-dimension condition can hold for any sub-Riemannian manifold equipped with a positive smooth measure. As a side note, we refer the reader interested in the study of curvature-dimension in the sub-Finsler setting to the works \cite{borzatashiro}, \cite{magnabosco2023failure}, \cite{borzatashiroSIU2024n1} and \cite{borzatashiroSIU2024n2}.

Surprisingly, it was discovered in \cite{panwei,Pan23} that it suffices to consider a sub-Riemannian structure on a manifold with boundary and equip it with a smooth measure that vanishes on the boundary points to construct an example of a $\mathsf{CD}$-space in sub-Riemannian geometry.
% {\blue Surprisingly, contrary to the expectation from researchers from Riemannian/sub-Riemannian geometry, Pan--Wei constructed an example of $\RCD(0,N)$ space which admits sub-Riemannian structure on its boundary points.}\todo{Let's discuss this sentence}
This example, which we revisit here, is also discussed in \cite{rizzi2023failure}. For $\alpha \geq 0$, the $\alpha$-Grushin plane $\mathbb{G}_\alpha$ is the sub-Riemannian structure on $\R^2$ induced by the vector fields
\[
X = \partial_x,~~~~Y^\alpha = \abs{x}^\alpha \partial_y.
\]
This defines a metric space $(\mathbb{G}_\alpha, \diff_{\mathbb{G}_\alpha})$ that admits the metric tensor
\[g_{\G_\alpha}=\diff x\otimes \diff x+\frac{1}{|x|^{2\alpha}}\diff y\otimes \diff y,\]
at non-singular points.
In \cite{panwei,rizzi2023failure}, the authors equip this metric space with the
the $\beta$-weighted measure $\mathfrak{m}_\beta = \abs{x}^{\beta - \alpha} \diff x \diff y$. Note that this weighted measure vanishes on the singular set $\{x = 0\}$ if $\beta>\alpha$, it coincides with the Lebesgue measure if $\beta=\alpha$,
and it fails to be locally finite if {\blue $\beta\leq \alpha-1$}. Although \cite{panwei} and \cite{rizzi2023failure} examined the same space, their techniques are different. In \cite{panwei}, Pan--Wei constructed a Riemannian manifold $M := [0,+\infty) \times_{f} \S^{n-1} \times_{g} \S^1$ with a doubly warped product metric such that
% {\red the contribution of the metric tensor over $\S^{n-1}$ is huge, and that over $\S^1$ is small for large $r \in [0,+\infty)$}
$f(r)\sim\sqrt{r}$ and $g(r)\sim r^{-2\alpha}$ as the first factor $r$ goes to $+\infty$.
It can be verified that $M$ has positive Ricci curvature.
By taking the asymptotic cone of the universal cover based at a specific point,
it can be argued that its (collapsed) limit is a Ricci limit space satisfying $\RCD(0,n+1)$.
Montgomery pointed out that this limit space is actually the  $\alpha$-Grushin half-plane.

In \cite{rizzi2023failure},
Rizzi--Stefani revisited this example by showing that the half-plane $\overline{\mathbb{G}}_\alpha^+ := \{ x \geq 0\}$ and the open half-plane $\mathbb{G}_\alpha^+ := \{ x > 0\}$ are both geodesically convex subsets of $\mathbb{G}_\alpha$ and that the weighted incomplete Riemannian manifold $(\mathbb{G}_\alpha^+, \diff_{\mathbb{G}_\alpha}, \mathfrak{m}_{\beta})$ satisfies the condition $\mathrm{Ric}_{N,V} \geq K$ if and only if
\begin{equation}
    \label{eq:KNagrushinhalfplane}
    K\leq 0,~~\text{and}~~-\alpha^2-\alpha+\beta\cdot \min\left(\alpha,1-\frac{\beta}{N-2}\right)\geq 0,
\end{equation}
where $V(x,y)=-\beta\log|x|$ is a singular potential.
Actually, \cite{rizzi2023failure} only considered the case $\alpha = 1$ but the general case $\alpha \geq 0$ is obtained in the same way.
%{\red Using a limiting argument, the $K$-convexity of the $N$-Rényi entropy functional on $\overline{\G}_\alpha^+$ can be inherited from $\G^+_\alpha$,  and thus }
They showed that the $\alpha$-Grushin half-plane $\overline{\G}_\alpha^+$ verifies $\cd(K,N)$ for the values of $K \in \R$ and $N \in (2, +\infty]$ such that \cref{eq:KNagrushinhalfplane} holds, using the geodesic convexity and the weighted Ricci curvature lower bound on the interior $\G_\alpha^+$. They also proved that $\overline{\G}_\alpha^+$ is infinitesimally Hilbertian, and thus the $\mathsf{RCD}(K, N)$ is also satisfied.
% {\red The proof of \cite[Theorem 1.11]{rizzi2023failure} which we generalise in \cref{theorem:equivalence} then shows that these values of $N$ and $K$ are exactly the ones for which the $\mathsf{RCD}(K, N)$ condition holds in the sub-Riemannian manifold (with boundary) $(\overline{\mathbb{G}}_\alpha^+, \diff_{\mathbb{G}_\alpha}, \mathfrak{m}_{\beta})$. }

The goal of this paper is to provide a few more examples of this phenomenon. We introduce a sub-Riemannian structure on the hemisphere and on the hyperbolic half-plane that depend on a parameter $\alpha \geq 0$, and by tweaking their unbounded ``Riemannian'' volume measure with a parameter $\beta \geq \alpha$, we obtain metric measure spaces that are $\mathsf{RCD}$-spaces. One motivation for studying these kinds of singular spaces is the singular Weyl's laws studied in \cite{boscain-prandi-seri2016, Chitour-Prandi-Rizzi24, honda2023} and the references therein.

In \cref{section:alphaGrushin}, we define the $\alpha$-Grushin hemisphere and its $\beta$-weighted measure. This metric measure space is an example of sub-Riemannian manifold where the $\mathsf{RCD}(K, N)$ condition holds for some $K > 0$.

\begin{theorem}
\label{thm:main1}
The metric measure space consisting of the sub-Riemannian $\alpha$-Grushin hemisphere equipped with its $\beta$-weighted measure satisfies the following properties.
\begin{enumerate}[label=\normalfont(\roman*)]
    \item For any $\alpha\geq 0$ and $K>0$, there is $\beta\geq\alpha$ such that the $\RCD(K,\infty)$ condition is satisfied.
    \item For any $\alpha\geq 1$ and $N\geq 2+4\alpha(\alpha+1)$, there is $\beta>\alpha$ such that $\RCD(0,N)$ holds.
    \item When $\RCD(0,N)$ is satisfied for some $N \in \linterval{2}{+\infty}$, there is $K>0$ such that $\RCD(K,N)$ holds.
\end{enumerate}
\end{theorem}
Similarly, we introduce in \cref{section:alphaGrushin2} the $\alpha$-Grushin hyperbolic half-plane and its $\beta$-weighted measure. This metric measure space is an example of sub-Riemannian manifold where the $\mathsf{RCD}(K, N)$ condition holds for some $K < 0$.

\begin{theorem}
\label{thm:main2}
    The metric measure space consisting of the sub-Riemannian $\alpha$-Grushin hyperbolic half-plane equipped with its $\beta$-weighted measure satisfies the following properties.
    \begin{enumerate}[label=\normalfont(\roman*)]
    \item For any $\alpha\geq 1$ and $N\geq 2+4\alpha(\alpha+1)$,
    there are $\beta>\alpha$ and $K<0$ such that $\RCD(K,N)$ holds.
    \item For any $\alpha,\beta\geq0$ and $N\in(2,+\infty)$, it does not satisfy $\RCD(0,N)$. Furthermore, it satisfies $\RCD(0,+\infty)$ if and only if $\alpha=1$ and $\beta\geq 2$.
    \item For any $\alpha\geq 1$,
        there are $\beta>\alpha$ and $N\in(-\infty,0)$ such that $\cd(0,N)$ holds.
    % \todo{2. It is possible to define $\RCD(K,N)$ for negative $N$, but since we do not know the equivalence of $\BE$ and $\cd$ for negative $N$, it does not give further information. Ask Chiara.}
\end{enumerate}
\end{theorem}
The negativity of $K$ is essential in the sense that the $\alpha$-Grushin hyperbolic half-plane satisfies the $\cd(0,N)$ condition only if $N$ is negative or $+\infty$.
An odd phenomenon occurs at $\alpha=1$, where the $1$-Grushin hyperbolic half-plane satisfies $\RCD(0,+\infty)$. It is unclear if this phenomenon is relevant to the validity of hyperbolicity or $\mathsf{CAT}(0)$ property in the $1$-Grushin hyperbolic plane.
\begin{remark}
In \cite{Pan23}, Pan constructed a sequence of doubly warped metric spaces which collapses to the $1$-Grushin hemisphere. We do not know if a similar construction holds for other $\alpha$-Grushin spaces, nor if the limit measure in such a construction would be equal to our $\beta$-weighted measure.
\end{remark}

As another generalisation of the $\alpha$-Grushin plane, we introduce a new sub-Riemannian manifold with infinite Hausdorff dimension, which we call the $\infty$-Grushin plane.
The description of this space is found in \cref{section:alphaGrushin3}.
By tweaking the Riemannian volume measure by multiplicative factors that depend on two parameters $\beta$ and $\gamma$,
the $\infty$-Grushin half-plane is shown to satisfy the $\RCD(0,+\infty)$.
% {\red We wished to also give an example of a true $\mathsf{RCD}(0, \infty)$ in sub-Riemannian geometry, and this is realised with the $\alpha$-Grushin half-plane equipped with a weighted measure that depends on two parameters $\beta$ and $\gamma$. The description of this space is found in \cref{section:alphaGrushin3}.}

\begin{theorem}
\label{thm:main3}
   The metric measure space consisting of the sub-Riemannian $\infty$-Grushin half-plane equipped with its $(\beta, \gamma)$-weighted measure satisfies the following properties.
\begin{enumerate}[label=\normalfont(\roman*)]
    \item There are $\beta,\gamma\geq 0$ such that $\RCD(0,+\infty)$ holds.
    \item For any $\beta,\gamma\geq 0$ and $K\in \R$, there is no $N\in(2,+\infty)$ such that $\RCD(K,N)$ holds.
\end{enumerate}
\end{theorem}

% {\blue
% \begin{remark}
%     In \cite{Pan23}, Pan constructed a sequence of doubly warped metric spaces which collapses to $1$-Grushin hemisphere.
%     We do not know if a similar construction holds for other $\alpha$-Grushin model spaces,
%     nor the limit measure in such a construction is equal to our weighted measure.
% \end{remark}}\todo{I don't think this remark is relevant here, maybe in another section}

  \begin{remark}\label{rmk:compact}
    The following consequence is to be noted.
\cref{thm:main3} implies that, for any $N\geq 2$, Gromov--Hausdorff limits of a sequence of metric measure spaces in the class
    \[\mathcal{X}_N:=\left\{(X,\di,\m)\mid \text{proper, }\RCD(0,\infty) \text{ and } \dim_H(X,\di)\leq N\right\}\]
    can have infinite Hausdorff dimension, i.e. $\overline{\mathcal{X}_N}\not\subset \bigcup_{n\in\N}\mathcal{X}_n$.
    Indeed, for $\epsilon>0$, the subsets $\R_{\geq \epsilon}\times \R$ of the $\infty$-Grushin half-plane are weighted Riemannian manifolds with boundary contained in $\mathcal{X}_2$,
    which converges to the $\infty$-Grushin half-plane as $\epsilon\to 0$, which has infinite Hausdorff dimension (see \cref{lemma:infinitehausdorffdimension}).
\end{remark}
%\todo{Sam: I don't like itemize in the remark, it doesn't look good. Before the first bullet was outside, we could go back to that. We could also have two remarks.\\
%K: I put the first remark after hyperbolic. I think this is the correct space. S: PErfect!}

The strategy of the proof of the previous results is fairly simple. After studying these spaces in detail in \cref{section:themodels}, we show in \cref{section:equivalence} that, under the right conditions of smoothness and geodesic convexity (see \cref{theorem:equivalence}), the validity of the $\mathsf{CD}(K, N)$ condition for a metric measure space whose interior is Riemannian is equivalent to the bound $\mathrm{Ric}_{N,V} \geq K$ on its interior. The values of $(K, N)$ for which the $\mathsf{CD}(K, N)$ and $\mathsf{RCD}(K, N)$ conditions holds in the $\alpha$-Grushin half-spaces are then just a matter of Ricci curvature computations, which we provide in \cref{section:riccicomputation}.

\section*{Acknowledgments}
This project has received funding from the European Research Council (ERC) under the European Union’s Horizon 2020 research and innovation program (grant agreement No. 945655).
The authors thank Shouhei Honda for leading the authors' attention to these problems, and telling us the consequences in \cref{rmk:compact}.

\section{Preliminaries}
\subsection{The \texorpdfstring{$\cd(K,N)$}{CD(K, N)} and \texorpdfstring{$\RCD(K,N)$}{RCD(K, N)}  conditions}
Firstly,
we start by recalling the curvature-dimension condition $\cd(K,N)$ and Riemannian curvature-dimension condition $\RCD(K,N)$.
A metric measure space is a triple $(\X,\diff,\m)$ where $(\X,\diff)$ is a complete, separable, locally compact and geodesic metric space, and $\m$ is a non-negative Radon measure on it.
We emphasize that the metric measure spaces will always be assumed to be essentially non-branching.
We denote by $C([0, 1], \X)$ the space of continuous curves from $[0, 1]$ to $\X$, and for $s \in [0, 1]$, we let $e_s : C([0, 1], \X)\ni \gamma\mapsto \gamma(s) \in \X$ be the evaluation map.
A geodesic $\gamma:[0,1]\to \X$ is a curve such that $\diff(\gamma(s), \gamma(t)) = |t - s| \diff(\gamma(0), \gamma(1))$ for all $s, t \in [0, 1]$, and we denote by $\Geo(\X)$ the space of all geodesics on $(\X,\diff)$.
Furthermore, the set of Borel probability measures on $\X$ is denoted by $\Prob(\X)$ and the set of those having finite second momentum by $\Prob_2(\X) \subseteq \Prob(\X)$. We endow the space $\Prob_2(\X)$ with the Wasserstein distance $W_2$, defined by
\begin{equation}\label{wasserstein}
W_2^2(\mu_0, \mu_1) := \inf_{\pi \in \mathsf{Adm}(\mu_0,\mu_1)}  \int \diff^2(x, y) \, \diff \pi(x, y),
\end{equation}
where $\mathsf{Adm}(\mu_0, \mu_1):=\{\pi\in \X\times \X\mid (\p_i)_\sharp \pi = \mu_i,\ \p_i : \text{the projection to the $i$-th factor}~(i=1,2)\}$.
The metric space $(\Prob_2(\X),W_2)$ is itself complete, separable and geodesic.
A probability measure $\pi\in \prob(\X\times \X)$ which attains the minimum values in \cref{wasserstein} is called an optimal transport plan.

\paragraph{The \texorpdfstring{$\cd(K,N)$}{CD(K, N)} condition for positive \texorpdfstring{$N$}{N}.}
In this context, the curvature-dimension condition is defined as follows.
For every $K \in \R$, $N\in [1,\infty)$ and $t\in [0,1]$, the \emph{distortion coefficients} are the functions

% \todo{5. wrong. Sam Check. DONE}
\begin{equation*}
\tau_{K,N}^{(t)}(\theta):=
\begin{cases}

\displaystyle  +\infty & \textrm{if}\  (N-1)\pi^{2}\leq K\theta^{2} ~\text{ and }~K\theta^2>0,\\
\displaystyle t^{1/N}\left(\frac{\sin(t\theta\sqrt{K/N})}{\sin(\theta\sqrt{K/N})}\right)^{1-1/N} & \textrm{if}\  0 < K\theta^{2} < (N-1)\pi^{2},\\
t & \textrm{if}\ %K \theta^{2}<0 \ \textrm{and}\ N=0, \ \textrm{or  if}\
K\theta^2 =0, ~~\text{or}~~K\theta^2<0~\text{ and }~N=1, \\
\displaystyle  t^{1/N}\left(\frac{\sinh(t\theta\sqrt{-K/N})}{\sinh(\theta\sqrt{-K/N})}\right)^{1-1/N} & \textrm{if}\ K\theta^2 < 0~\text{ and }~N>1.
\end{cases}
\end{equation*}

\begin{definition}[{$\cd(K,N)$ condition for $N\in[1,+\infty]$, \cite{sturm2006, lott--villani2009}}]\label{def.2.1}
% \todo{6. SAM $N=1$? $t\in[0,1] $ a.e.? $N=+\infty$ case? DONE.}
For $K\in\R$ and $N\in[1,+\infty]$, a metric measure space $(\X,\di,\m)$ is said to satisfy the $\cd(K,N)$ condition if for every pair of absolutely continuous measures $\mu_0=\rho_0\m,\mu_1= \rho_1 \m \in \Prob_2(\X)$, there exists an optimal transport plan $\pi\in \prob(\X\times \X)$ and an absolutely continuous $W_2$-geodesic $(\rho_t\m)_{t\in [0,1]}$ connecting them such that the following inequality holds for every $N' \geq N$ and every $t \in [0,1]$:
\begin{equation*}\label{eq:CDpositive}
    \int_\X \rho_t^{1-\frac 1{N'}} \diff \m \geq \int_{\X \times \X} \Big[ \tau^{(1-t)}_{K,N'} \big(\diff(x,y) \big) \rho_{0}(x)^{-\frac{1}{N'}} +    \tau^{(t)}_{K,N'} \big(\diff(x,y) \big) \rho_{1}(y)^{-\frac{1}{N'}} \Big]    \diff \pi( x,y),
\end{equation*}
if $N < +\infty$, and, if $N = +\infty$,
  % \[
  %   \mathrm{Ent}(\mu_t | \m) \leq (1 - t) \mathrm{Ent}(\mu_0|\m) + t \mathrm{Ent}(\mu_1 | \m) - \frac{K}{2} t(1 - t) W_2^2(\mu_0, \mu_1).
  % \]
 \[
    \int_\X\rho_t\log\rho_t \diff\m \leq (1 - t)\int_\X\rho_0\log\rho_0 \diff\m + t \int_\X\rho_1\log\rho_1 \diff\m - \frac{K}{2} t(1 - t) W_2^2(\mu_0, \mu_1).\]
\end{definition}
\noindent It is not difficult to see that if a metric measure space satisfies the $\mathsf{CD}(K, N)$ condition for some $K \in \R$ and $N \in [1, +\infty]$, then it also satisfies the $\mathsf{CD}(K', N')$ for any $K' \leq K$ and $N' \in [N,+\infty]$, see \cite[Prop. 1.6]{sturm2006}.

% \todo{Sam: add the $N = \infty$ case\\
% Ken: Is this same?}

\paragraph{The \texorpdfstring{$\cd(K,N)$}{CD(K, N)} condition for negative \texorpdfstring{$N$}{N}.}
The curvature-dimension condition $\cd(K,N)$ for negative $N$ is defined in a similar way.
For $K\in\R$ and $N<0$,
 we set the distortion coefficients
\begin{equation*}
\tilde{\tau}_{K,N}^{(t)}(\theta):=
\begin{cases}

\displaystyle  +\infty & \textrm{if}\  (N-1)\pi^{2}\geq K\theta^{2},\\
\displaystyle t^{1/N}\left(\frac{\sin(t\theta\sqrt{K/N})}{\sin(\theta\sqrt{K/N})}\right)^{1-1/N} & \textrm{if}\  0>K\theta^{2} > (N-1)\pi^{2},\\
t & \textrm{if}\ %K \theta^{2}<0 \ \textrm{and}\ N=0, \ \textrm{or  if}\
K\theta^2 =0, \\
\displaystyle  t^{1/N}\left(\frac{\sinh(t\theta\sqrt{-K/N})}{\sinh(\theta\sqrt{-K/N})}\right)^{1-1/N} & \textrm{if}\ K\theta^2 > 0.
\end{cases}
\end{equation*}

% \todo{8. Do we need negative $N$?}
\noindent The following definition was first introduced in \cite{Ohta-negativeCD}.
\begin{definition}[{$\cd(K,N)$ condition for $N\in(-\infty,0)$}]
% \todo{Sam: A small paragraph about the relevance of CD with $N$ negative + standard properties like consistency + monotonicity. DONE.}
For $K\in\R$ and $N\in(-\infty,0)$,
a metric measure space $(\X,\di,\m)$ is said to satisfy the $\cd(K,N)$ condition if for every pair of absolutely continuous measures $\mu_0=\rho_0\m,\mu_1= \rho_1 \m \in \Prob_2(\X)$, there are an optimal transport plan $\pi\in\prob(\X\times \X)$ and an absolutely continuous a $W_2$-geodesic $(\rho_t\m)_{t\in [0,1]}$ connecting them such that the following inequality holds for every $N'\in[N,0)$ and every $t \in [0,1]$:
\begin{equation*}\label{eq:CDnegative}
    \int_\X \rho_t^{1-\frac 1{N'}} \diff \m \leq \int_{\X \times \X} \Big[ \tilde{\tau}^{(1-t)}_{K,N'} \big(\diff(x,y) \big) \rho_{0}(x)^{-\frac{1}{N'}} +    \tilde{\tau}^{(t)}_{K,N'} \big(\diff(x,y) \big) \rho_{1}(y)^{-\frac{1}{N'}} \Big]    \diff \pi( x,y).
\end{equation*}
\end{definition}
\noindent The $\cd(K,N)$ condition for negative $N$ has been found to appear naturally when studying harmonic measures on the sphere in \cite{Milman2017}, where the author uses previous results from \cite{Milman2017a} to obtain new isoperimetric inequalities for these harmonic measures. This condition is further studied in works such as \cite{Magnabosco2023b, Magnabosco2023a, Magnabosco2023c}. We also have the following consistency property: if a metric measure space satisfies the $\mathsf{CD}(K, N)$ condition for some $K \in \R$ and $N < 0$, then it also satisfies the $\mathsf{CD}(K', N')$ condition for every $K' \leq K$ and $N' \in [N, 0)$. We also have that the $\mathsf{CD}(K, +\infty)$ condition implies the $\mathsf{CD}(K, N)$ condition for any $N < 0$, see \cite[Lemma 2.9]{Ohta-negativeCD}.

\paragraph{The \texorpdfstring{$\RCD(K,N)$}{RCD(K, N)} condition.}
In addition to the $\cd(K,N)$ condition,
we need to introduce infinitesimal Hilbertianity to define $\RCD(K,N)$ spaces, see \cite{Gigli12}.
On a proper metric measure space $(\X,\diff,\m)$,
the Cheeger energy $\mathrm{Ch}:L^2(\X,\m)\to \R$ is defined by
\[\mathrm{Ch}(f):=\inf\left\{\liminf_{n\to \infty}\frac{1}{2}\int_\X\Lip(f_n)\diff \m\mid f_n\in\Lip_b(\X,\diff)\cap L^2(\X,\m),\|f_n-f\|_{L^2}\to 0\right\},\]
where $\Lip_b(\X,\diff)$ is the space of bounded Lipschitz functions and for $f\in\Lip_b(\X,\diff)$ and $x\in \X$,
\[\Lip(f)(x):=\limsup_{y\to x}\frac{|f(y)-f(x)|}{\diff(x,y)}.\]
Then the space of functions $H^{1,2}(\X,\diff,\m):=\{f\in L^2(\X,\m)\mid \mathrm{Ch}(f)<+\infty\}$ defines a Banach space with the norm $\|f\|_{H^{1,2}}:=\left(\|f\|_{L^2}^2+2\mathrm{Ch}(f)^2\right)^{1/2}$.

There are different equivalent definitions of infinitesimal Hilbertianity, and the following is the one we choose here, see \cite{Gigli12} and \cite{Ambrosio2013} for more details.

\begin{definition}[Infinitesimal Hilbertianity and $\RCD(K,N)$ condition]
    A metric measure space $(\X,\diff,\m)$ is said to be infinitesimally Hilbertian if $(H^{1,2}(\X,\diff,\m),\|\cdot\|_{H^{1,2}})$ is a Hilbert space.

    Furthermore, given $K\in\R$ and $N\in[1,+\infty]$, the metric measure space $(\X,\diff,\m)$ satisfies the $\RCD(K,N)$ condition if it verifies the $\cd(K,N)$ condition and is infinitesimally Hilbertian.
\end{definition}

%{\blue }\todo{K: I think we can skip this sentence since we mention it in the next remark. Sam: I think we should keep it otherwise it's two definitions in a row which is pretty ugly.}

%\begin{definition}[{$\RCD(K,N)$ condition}]
%    For $K\in\R$ and $N\in[1,+\infty]$, a metric measure space $(\X,\diff,\m)$ satisfies the $\RCD(K,N)$ condition if it satisfies the $\cd(K,N)$ condition and is infinitesimally Hilbertian.
%\end{definition}

\begin{remark}
    We do not speak about the $\mathsf{RCD}$ condition for $N < 0$ for a few reasons. Firstly, the $\mathsf{CD}$ condition for negative $N$, contrary to when $N \geq 1$, allows for measures $\mathfrak{m}$ that are not locally finite. If that's the case, it is not known if the equivalences of the weak gradients in \cite[Sections 7 and 8]{Ambrosio2013} still hold. Secondly, one of the reasons the $\mathsf{RCD}(K, N)$ condition is  particularly successful is because it is equivalent to the (synthetic) Bakry-\'Emery condition $\mathsf{BE}(K, N)$ under the Sobolev-to-Lipschitz property, see \cite{Erbar2015} and \cite{Ambrosio2015}.
    %{\red  and with the entropic curvature dimension condition $\cd^e(K,N)$,} see \cite{Erbar2015, Ambrosio2015}
    %This equivalence is not yet known for $N < 0$. We refrain from using the $\RCD(K,N)$ condition for negative $N$ because of the lack of a unified treatment in the literature.
    In our specific setting, the measures are always nonnegative, and we only use the equivalence between $\cd(K,N)$ and the smooth Ricci lower bound $\mathrm{Ric}_{N,V} \geq K$.
    % \todo[inline]{Bakry--Emery condition is
    % \[|\nabla P_tu|^2 + \frac{1-e^{-2Kt}}{NK}({\rm L}P_t u)^2\le e^{-2Kt}P_t|\nabla u|^2,\]
    % where $P_t u $ is the solution to the heat equation with the initial value condition $P_0 u = u$. Sam: Ok. I meant the Bakry-Emery tensor.\\
    % K:Then, how about using the notation Ricci lower bound?} This probably mean that any reasonable definition of $\mathsf{RCD}(K, N)$ for $N < 0$ would be applicable in our case.}
    %This probably means that any reasonable definition of $\mathsf{RCD}(K, N)$ for $N < 0$ would be applicable in our case.
    Once the gaps in the $\mathsf{RCD}$ theory for negative effective dimension will have been filled, one should be able to replace $\mathsf{CD}(K, N)$ with $\mathsf{RCD}(K, N)$ when $N < 0$ in this paper, e.g. in (iii) of \cref{thm:main2}
\end{remark}

\subsection{Sub-Riemannian geometry}
\label{subsection:subR}

Before we introduce the $\alpha$-Grushin spaces, we recall some basic facts about sub-Riemannian and metric geometry. For a more comprehensive account of these topics, we refer the reader to \cite{ABB-srgeom} and \cite{burago-burago-sergei2001}, for example.
The $\alpha$-Grushin spaces considered in this work will all be two-dimensional almost-Riemannian manifolds and we recommend especially \cite[Chapter 9]{ABB-srgeom}.

A sub-Riemannian structure on an $n$-dimensional manifold $M$ is given by a set of $m$ globally defined vector field $\mathcal{F} := \{X_1, \dots, X_m\}$, also called the generating frame. The associated distribution is defined as the family, indexed with $x \in M$, of the vector subspaces
\[
\mathcal{D}_x := \mathrm{span} \{X_1(x), \dots, X_m(x)\} \subseteq T_xM.
\]
From this data, it is possible to introduce an inner product $g_x$ on $\mathcal{D}_x$ by applying the polarisation formula to
\begin{equation}
    \label{eq:subRmetric}
    g_x(v, v) := \inf \left\{ \sum_{i = 1}^m u_i^2 \mid \sum_{i = 1}^m u_i X_i(x) = v  \right\}.
\end{equation}

The rank of the sub-Riemannian structure at $x \in M$ is defined by $r(x) := \dim (\mathcal{D}_x)$. A two-dimensional almost-Riemannian manifold is a sub-Riemannian structure on a two-dimensional manifold $M$ such that the cardinal of $\mathcal{F}$ is two.
% \todo{10. DONE. meaning of $|\mathcal{F}|$. Maybe we have to mention a general definition of almost Riemannian manifolds}.
In a two dimensional sub-Riemannian structure, a point $x \in M$ such that $r(x) = 2$ (resp. $r(x) = 1$) is called a Riemannian point (resp. singular point). The singular set, i.e. the set of singular points, must necessarily be small (see \cite[Section 9.1.1]{ABB-srgeom}). In a neighborhood of Riemannian points,
% \todo{11. DONE. In a neighborhood of Riemannian points,? I feel it is a bit dangerous}
the inner product \cref{eq:subRmetric} is a well-defined metric tensor and we can introduce, at those points, a Riemannian volume density (using the same formula as in Riemannian geometry) which diverges when approaching a singular point.

An admissible (or horizontal) curve $\gamma : \interval{0}{1} \to M$ is an absolutely continuous path such that there exists a control $u \in \mathrm{L}^2(\interval{0}{1}, \R^m)$ satisfying
% \todo{12. $T=1$? DONE.}
\[
\dot{\gamma}(t) = \sum_{i = 1}^m u_i(t) X_i(\gamma(t)), \ \text{ for almost every } t \in \interval{0}{1}.
\]
The sub-Riemannian length of an admissible curve $\gamma : \interval{0}{1} \to M$ is then defined by
% \todo{ 13. scalar product was missing DONE}
\[
\mathrm{Length}(\gamma) := \int_0^1 \sqrt{g_{\gamma(t)}(\dot{\gamma}(t), \dot{\gamma}(t))} \mathrm{d}t,
\]
and the sub-Riemannian distance between two points $x, y \in M$ is
\begin{equation}\label{e:distance}
\mathrm{d}(x, y) := \inf \left\{ \mathrm{Length}(\gamma) \mid \gamma \text{ admissible and joins } x \text{ and } y \right\}.
\end{equation}
It is not always given that $\mathrm{d}$ really defines a distance function. If $\mathcal{F}$ satisfies the bracket-generating condition (see \cite[Definition 3.1]{ABB-srgeom}), for instance, then Rashevskii-Chow theorem \cite[Section 3.2]{ABB-srgeom} implies that there exists an admissible curve between every two points of $M$, and that $(M, \diff)$ is a metric space with the metric and manifold topology coinciding.

% Denoting by $\pi: \T^\ast(\mathds{H}) \to \mathds{H}$ is the canonical bundle projection, the Hamiltonian of the sub-Riemannian structure is the map $\mathcal{H} : \T^*(\mathds{H}) \times \mathds{R}^m \times \{0, -1\} \to \mathds{R}$ defined by
% \[
% \mathcal{H}(\lambda, u, \nu) := \sum_{i = 1}^m u_i \langle \lambda_0, X_k(\pi(p)) \rangle + \frac{\nu}{2} \sum_{i = 1}^m u_i^2,
% \]
% while the maximised Hamiltonian is given by
% \[
% H : \mathrm{T^*}M \to \mathds{R} : (p, \lambda_0) \mapsto H(p, \lambda_0) :=  \frac{1}{2}\sum_{k = 1}^m \langle \lambda_0, X_k(p) \rangle^2.
% \]

% The rank $r(x)$ of the sub-Riemannian structure at $x \in M$ is defined by $\dim \mathcal{D}_x$. A two-dimensional almost-Riemannian manifold is a sub-Riemannian structure on a two-dimensional manifold $M$ such that $|\mathcal{F}| = 2$. A point $x \in M$ such that $r(x) = 2$ (resp. $r(x) = 1$) is called a Riemannian point (resp. singular point). The singular set, i.e. the set of singular point, must necessarily be small (see \cite[Section 9.1.1]{ABB-srgeom}). At Riemannian points, we have a well-defined metric tensor and Riemannian volume measure which diverge when approaching a singular point.

Denoting by $\pi: T^\ast M \to M$ is the canonical bundle projection, the Hamiltonian is the map $H : T^*M \to \R$ defined by
\[
H(\lambda) :=  \frac{1}{2}\sum_{k = 1}^m \langle \lambda, X_k(\pi(\lambda)) \rangle^2 ~~~~\forall \lambda\in T_xM
\]
%{\red A (constant speed) length minimiser, or minimising geodesic, is a horizontal curve $\gamma : \interval{0}{1} \to M$ such that $\diff(\gamma(s), \gamma(t)) = |t - s| \diff(\gamma(0), \gamma(1))$ for all $s, t \in [0, 1]$.}
% \todo{15: Geodesics should be constant speed. DONE}
Pontryagin's Maximum Principle is very helpful in the search for
%{\red length-minimisers.}
geodesics.

\begin{theorem}[Pontryagin's Maximum Principle]
\label{theorem:PMP}
     If $\gamma$ is length minimiser parametrised by constant speed, then there exists a Lipschitz curve $\lambda(t) \in T^*_{\gamma(t)}M$ such that one and only one of the following is satisfied:
\begin{enumerate}[label=\normalfont(\roman*)]
\item $\dot{\lambda} = \overrightarrow{H}(\lambda)$, where $\overrightarrow{H}$ is the unique vector field in $T^*M$ such that $\sigma(\cdot, \overrightarrow{H}(\lambda)) = \mathrm{d}_\lambda H$ for all $\lambda \in T^*M$;
\item $\langle \lambda(t), X_i(\gamma(t)) \rangle = 0$ for all $i = 1, \dots, m$, and $\lambda(t) \neq 0$ for all $t \in \interval{0}{1}$.
\end{enumerate}
\end{theorem}
A curve $\lambda : \interval{0}{1} \to T^*M$ satisfying (i) (resp. (ii)) in the theorem above is called a normal (resp. abnormal) extremal. \cref{theorem:PMP} is thus stating that a (constant speed) minimising geodesic has a cotangent lift that is a normal or an abnormal extremal. Note that an extremal in a two-dimensional almost-Riemannian manifold is abnormal if and only if its projection is a constant curve that lies on the singular set (see \cite[
Theorem 9.2]{ABB-srgeom}).

For the remainder of this work, we will adopt the notation $(x)^{2 \alpha} := (x^2)^\alpha  \in \R_{\geq 0}$ for every $\alpha \geq 0$ and $x\in\R$.

\section{Geometry of \texorpdfstring{$\alpha$}{alpha}-Grushin half-spaces}
\label{section:themodels}

\subsection{The \texorpdfstring{$\alpha$}{alpha}-Grushin sphere and hemisphere}
\label{section:alphaGrushin}

On the two-dimensional Riemannian sphere $(\S^2, g_{\S^2})$, we introduce the following coordinate chart, the validity of which can be found in \cite{boscain-prandi-seri2016}.
% \todo{16. Since the equator is vertical, $N$ and $S$ should be $E$ and $W$, the east and west pole? Answer: $N, S$ do not belong to $\gamma$.}
Fix a large circle $\gamma:\R/2\pi\Z\to \S^2$,
and let $N$ and $S$ be the north pole and south pole of $\S^2$ respectively, with respect to $\gamma$. For $p\in \S^2\setminus\{N,S\}$,
we define $x := x(p) \in \ointerval{-\pi/2}{\pi/2}$ as the signed (spherical) distance $\diff_{\S^2}(p,\mathrm{Im}(\gamma))$,
where the sign is positive (resp. negative) if $p$ belongs to the hemisphere containing $N$ (resp. $S$).
Furthermore,
define the number $y=y(p)\in \R/2\pi\Z$ so that $\gamma(y)$ is the perpendicular foot from $p$ to $\mathrm{Im}(\gamma)$. The map $\S^2 \to \ointerval{-\pi/2}{\pi/2}\times \R/2\pi\Z  : p \mapsto (x(p), y(p))$ is a well-defined coordinate chart which can be naturally extended to $N$ and $S$ under the identification
$(\frac{\pi}{2},y_1) \sim (\frac{\pi}{2},y_2)$ and $(-\frac{\pi}{2},y_1) \sim (-\frac{\pi}{2},y_2)$ for any $y_1,y_2\in \R/2\pi\Z$. Note that $N$ (resp. $S$) corresponds to the equivalence class of $(\frac{\pi}{2},y)$ (resp. $(-\frac{\pi}{2},y)$). It is easily shown that in these coordinates, the spherical Riemannian metric tensor $g_{\S^2}$ possesses the warped product structure
\[
g_{\S^2} = \diff x \otimes \diff x + \cos^2(x) \diff y \otimes \diff y.
\]
\begin{remark}
% \todo{17. Do we use spherical coordinates? We choose it like this because it is exactly the same construction for the hyperbolic case, and justifies it. DONE.}
The coordinate system $(x, y)$ is, up to rotations, the same as the spherical coordinates $(\phi,\theta) \in  (-\pi/2, \pi/2)\times \R / 2\pi\Z$
% \todo{20. This should be open? DONE.}
which parametrises the standard sphere $\S^2 \setminus \{N, S\}\subseteq \R^3$ by
\[(\theta,\phi)\mapsto (\cos(\theta)\cos(\phi),
\sin(\theta)\cos(\phi),\sin(\phi)).\]
\end{remark}
With this in mind, we can introduce the $\alpha$-Grushin sphere and hemisphere.
\begin{definition}
% \todo{21. We should say somewhere that $\alpha$ can be non-integer, since the PMP holds, but it is no longer almost Riemannian manifolds}
\label{def:alphasphere}
    For $\alpha \geq 0$, the $\alpha$-Grushin sphere $\S_\alpha$ is the sub-Riemannian structure on $\S^2$ induced from the vector field $X$ and $Y^\alpha$ given by
\[
X:=\partial_{x},~~Y^\alpha:=\frac{\abs{\sin(x)}^{\alpha}}{\cos(x)}\partial_{y}.
\]
The $\alpha$-Grushin hemisphere $\overline{\S}^+_\alpha$ (resp. open hemisphere $\S^+_\alpha$) is the subset of $\S_\alpha$ defined by
\[
\overline{\S}^+_\alpha := \{ p \in \S^2 \mid x \in [0,\pi/2] \} \text{ (resp. $\S^+_\alpha := \{ p \in \S^2 \mid x \in (0,\pi/2] \})$}.
\]

\begin{remark}
    When $\alpha$ is non-integer, the vector fields are not smooth and they are not bracket-generating. However, any pair of points can still be joined with a horizontal curve and Pontryagin's Maximum Principle stated in \cref{theorem:PMP} can be applied. Therefore, the Carnot-Carathéodory metric \cref{e:distance} can be constructed as in smooth sub-Riemannian geometry. This remark remains valid for the other model spaces introduced in this work.
\end{remark}

\begin{remark}
     Note that, strictly speaking, the structure introduced in \cref{def:alphasphere} does not fall into the definition of sub-Riemannian structure laid out in \cref{subsection:subR}. Indeed, the vector fields $X$ and $Y^\alpha$ are not global vector fields: they are not defined at the poles, i.e. at $x=\pm\pi/2$. To be completely rigorous, one should therefore check that it is still a sub-Riemannian manifold but according to the general definition of \cite[Definition 3.2]{ABB-srgeom}. There, a sub-Riemannian manifold is given by couple $(E, f)$ where $E$ is a Euclidean vector bundle and $f : E \to T M$ is a morphism of vector bundles.
     In our setting, $E = T\Sbb^2$ and $f$ is a morphism that we now construct explicitly.

    Denote by $p:T\Sbb^2\to \Sbb^2$ the bundle projection. Letting $\mathcal{U}_1 := \mathbb{S}^2 \setminus \{N, S\}$, the coordinate chart
    $\phi_1 := (x, y) : \mathcal{U}_1 \to \mathbb{R}^2$
    induces a chart on $T\mathbb{S}^2$, and we define
    $f_{1} : p^{-1}(\mathcal{U}_1) \to T \mathbb{S}^2$
    as the map, linear on fibers, that satisfies
    \[
    f_{1}(\partial_x) = \partial_x, \qquad f_{1}(\partial_y) = \abs{\sin(x)}^{\alpha} \partial_y.
    \]
    Letting $\mathcal{U}_2 := \{ (s, t, z) \in \mathbb{S}^2\subset \R^3 \mid z > 0 \}$ the map
    \[
    \phi_2 : \mathcal{U}_2 \to \mathbb{R}^2 :  (s, t, \sqrt{1 - s^2 - t^2}) \mapsto (s, t)
    \]
    is another coordinate chart on $\mathbb{S}^2$, which also induces a chart on $T \mathbb{S}^2$. We set $f_{2} : p^{-1}(\mathcal{U}_2) \to T \mathbb{S}^2$
    the map, linear on fibers, satisfying
    \[
    f_2(\partial_s) = \frac{1}{s^2+t^2} \left[ \left(s^2+(1-s^2-t^2)^{\alpha/2}t^2\right)  \partial_s + st(1-(1-s^2-t^2)^{\alpha/2}) \partial_t \right]
    \]
    and
    \[
    f_2(\partial_t) = \frac{1}{s^2+t^2} \left[ st(1-(1-s^2-t^2)^{\alpha/2}) \partial_s + \left((1-s^2-t^2)^{\alpha/2}s^2 + t^2 \right) \partial_t \right].
    \]
    This ensures $f_1=f_2$ on $p^{-1}(\mathcal{U}_1\cap\mathcal{U}_2)$ via the coordinate transformation $s=\cos(x)\cos(y)$, $t=\cos(x)\sin(y)$. The apparent singularity at $(s,t)=(0,0)$ is removable since
\[
f_2(\partial_s)\to\partial_s,
\quad
f_2(\partial_t)\to\partial_t
\quad\text{as }(s,t)\to(0,0),
\]
so $f_2$ extends smoothly over the origin, giving the identity on the fiber at $N$. The maps $f_1$ and $f_2$ are patched together to define a smooth, globally defined bundle morphism \(f\colon E\to T\Sbb^2\), and the pair $(E,f)$ defines a sub‐Riemannian manifold in the sense of \cite[Definition 3.2]{ABB-srgeom}.

\end{remark}

% {\blue At the point $x=\pm\frac{\pi}{2}$, we cannot extend $X$ and $Y^\alpha$, however, the Riemannian metric is smoothly extended. Indeed, after transforming the coordinates $(x,y)$ into the canonical upper hemisphere coordinates $(s,t)\mapsto (s,t,\sqrt{1-s^2-t^2})$, we can  verify the smoothness at $(s,t)=(0,0)$.}
\end{definition}
The $0$-Grushin sphere is simply the two-dimensional Riemannian sphere $\S^2$. When $\alpha > 0$, the $\alpha$-Grushin sphere is a two-dimensional almost-Riemannian structure with $\{x = 0\}$ being its set of singular points. The Grushin sphere studied in \cite{boscain-prandi-seri2016,Pan23} corresponds to the $1$-Grushin sphere. At non-singular points, this sub-Riemannian structure admits the Riemannian metric
\begin{equation}
    \label{eq:SalphaRiemannianMetric}
    g_{\S_{\alpha}} = \diff x \otimes \diff x + \frac{\cos^2(x)}{\sin^{2\alpha}(x)} \diff y \otimes \diff y.
\end{equation}
A simple computation shows that the
% \todo{24. unweighted should be removed? DONE}
 Riemannian volume induced from \cref{eq:SalphaRiemannianMetric}, is given by
\[
\vol_{\S_\alpha}=\cos(x)\abs{\sin(x)}^{-\alpha}\diff x\diff y.
\]
% \todo{25. Should we use $dx\wedge dy?$ No because we say measure (=density) and not form, so no negative sign which there could be with wedge. DONE.}
We introduce the following weighted measure.
\begin{definition}
\label{def:MeasurealphaGrushinSphere}
For $\beta\geq \alpha$,
we consider the weighted measure given by
\begin{equation*}
    \label{eq:MeasurealphaGrushinSphere}
    \m^\beta_{\S_\alpha}:=\abs{\sin(x)}^\beta\vol_{\S_\alpha}=\cos(x)\abs{\sin(x)}^{\beta-\alpha}\diff x \diff y=e^{-V_{\S_\alpha}}\vol_{\S_\alpha},
\end{equation*}
where $V_{\S_\alpha}(x,y):=-\beta\log\abs{\sin(x)}$.
\end{definition}

% \todo[inline]{Sam: strictly speaking, is $(\S_\alpha, \diff_{\S_\alpha}, \mathfrak{m}^\beta_{\S_\alpha})$ a metric measure space even for $\beta < \alpha$ or only when $\beta \geq \alpha$?\\
% Ken: You are right, if $\beta-\alpha\in(-\infty,-1]$, then the measure is not locally finite. If $\beta-\alpha\in(-1,0)$, then it is locally finite (and Radon?) but possibly some regularity might fail. It would be safe to assume $\beta\geq \alpha$.}
\noindent The $\alpha$-Grushin hemisphere is a geodesically convex subset of $\S_\alpha$ and a
% {\red uniquely}
geodesic space when seen as a length subspace of $\S_\alpha$. This is made clear by the next result.

\begin{prop}
    \label{prop:geodesicspacealphaGrushinSphere}
     There is a
     % {\red unique}\todo{Sam: do we need to prove uniqueness of geodesics for our result? Ken: I think no}
     minimising geodesic contained within $\overline{\S}^+_\alpha$ that joins any two given points in $\overline{\S}^+_\alpha$. Furthermore, the $\alpha$-Grushin open hemisphere $\S_\alpha^+$ is a geodesically convex subset of $\overline{\S}^+_\alpha$ and has the structure of an incomplete weighted Riemannian manifold when equipped with the restriction of the measure $\mathfrak{m}^\beta_{\S_\alpha}$.
    % {\red In addition,
    % $\gamma$ is written by the image of the exponential map defined in \ref{expo}. Sam: Maybe we do not really use this information anymore?\\
    % Ken: Agree, maybe not. What we need is (strongly) geodesically convexity of $M$, and it can be proved without exponential map.}
\end{prop}

\begin{proof}

We start by noting that between every two points on the $\alpha$-Grushin sphere, there is indeed a horizontal path controlled by the vector fields $X$ and $Y^{\alpha}$ joining them. This means that the induced sub-Riemannian distance $\diff_{\S_\alpha}$ is well-defined and that $(\S_\alpha, \diff_{\S_\alpha})$ is a locally compact metric space. Even though the bracket generating condition is not verified when $\alpha \notin \N$, it is easy to see that the metric topology still coincides with the original topology of $\S^2$ by the monotonicity property of $\mathrm{d}_{\S_\alpha}$ with respect to $\alpha \geq 0$. Furthermore, any metric ball is compact and thus the metric space $(\S_\alpha, \diff_{\S_\alpha})$ is complete.

The sub-Riemannian Hamiltonian $H : T^*(\S_\alpha) \to \R$ can be written in the canonical coordinates $(x, y, u,v )$ induced from $(x, y)$  as
\[
H(\lambda):=\frac{1}{2}\left[\langle \lambda,X\rangle^2+\langle\lambda,Y^\alpha\rangle^2\right]=\frac{1}{2}\left[u^2+\frac{\sin^{2\alpha}(x)}{\cos^2(x)}v^2\right].
\]
A normal extremal $\lambda : \interval{0}{T} \to T^*(\S_\alpha) : t \mapsto (x(t), y(t), u(t), v(t))$ satisfies the following Hamiltonian system of equations
\begin{equation}
    \label{eq:alphaGrushinsphereHamilton}
    \begin{cases}
    \dot{x}=u,\\
    \displaystyle \dot{y}=\frac{\sin^{2\alpha}(x)}{\cos^2(x)}v,\\
    \displaystyle \dot{u}= -  v^2 \  \sin^{2(\alpha - 1)}(x) \tan(x) (\alpha + \tan^2(x)),\\
    \dot{v}=0.
\end{cases}
\end{equation}
% {\red By standard results in the theory of ordinary differential equations \cite{abraham-marsden1978}, we easily get that the Hamiltonian is a complete vector field on $\mathrm{T}^*(\S_\alpha)$.\todo{Sam: need to find the right necessary condition.}
% The sub-Riemannian Hopf-Rinow theorem (see \cite[Theorem 11.38]{ABB-srgeom}) thus implies that the metric space $(\S_\alpha, \diff_{\S_\alpha})$ is complete. Therefore,}
Here note that the extremal reaches to the undefined point $x=\pm \pi/2$ only if $v\equiv 0$ (this follows from the non-integrability of $\tan(x)$ near $x=\pm\pi/2$). In this case, a (Euclidean) large circle passing through the north pole becomes a length minimizing geodesic. By completeness, there is a sub-Riemannian geodesic between every two points of $\S_\alpha$ by \cite[Theorem 2.5.23]{burago-burago-sergei2001}. These are obtained from Hamilton's equation \cref{eq:alphaGrushinsphereHamilton} since there are no non-trivial abnormal geodesics.

If a horizontal path of $\Sbb_\alpha$ is contained in both $\overline{\S}^+_\alpha$ and $\overline{\S}^-_\alpha := \{ p \in \S^2 \mid x \in [-\pi/2,0] \}$,
then a reflection $(x, y) \mapsto (-x,y)$ of the part of path that is in $\overline{\S}^-_\alpha$ produces a curve contained in $\overline{\S}^+_\alpha$ with the same length. This shows that a geodesic between points in the $\alpha$-Grushin hemisphere $\overline{\S}^+_\alpha$ is contained within $\overline{\S}^+_\alpha$. Length-minimisers are smooth since they satisfy Hamilton's equation \cref{eq:alphaGrushinsphereHamilton}. Thus, a constant-speed minimising geodesic $\gamma(t) = (x(t), y(t))$ that touches the singular equator at a point other than its endpoints must do so tangentially, and \cref{eq:alphaGrushinsphereHamilton} implies that $x(t)$ vanishes for all $t$. Consequently, $y(t)$ also vanishes, and $\gamma$ becomes a constant curve. In particular, a minimising geodesic between points of $\S_\alpha^+$ is also contained within $\S_\alpha^+$.

The fact that $\S_\alpha^+$ is also an incomplete Riemannian manifold follows easily since it doesn't contain any singular points of $\S_\alpha$.
% }
\end{proof}

\subsection{The \texorpdfstring{$\alpha$}{alpha}-Grushin hyperbolic plane and half-plane}\label{section:alphaGrushin2}

On the two-dimensional hyperbolic plane $(\bbH^2, g_{\bbH^2})$, we consider the following coordinate chart, called Lobachevsky's coordinates (see \cite[Section 33.1]{Martin-book}, for example).
%\todo{maybe we add a reference for that Ken: Note tha this coordinates is called the Lobachevsky coordinates, but not introduced by him. I think we should find a reference other than Russian, but it looks nothing}
We fix an infinite minimising geodesic ray $\gamma : \R \to \bbH^2$, and for $p \in \bbH^2$, we let $x := x(p) \in \R$ be signed hyperbolic distance $\diff_{\bbH^2}(p, \mathrm{Im}(\gamma))$,
where the signature is positive (resp. negative) if $p$ belongs to the left hand side (resp. right hand side) of $\gamma$. Furthermore, let $y=y(p)\in \R$ be the unique number such that $\gamma(y)$ is the perpendicular foot from $p$ to $\mathrm{Im}(\gamma)$. The map $\bbH^2 \to \R \times \R : p \mapsto (x(p), y(p))$ defines a global coordinate chart, and a short computation shows that the hyperbolic Riemannian metric $g_{\bbH^2}$ has the warped product structure
\[
g_{\bbH^2} = \diff x\otimes \diff x+\cosh^2(x)\diff y\otimes \diff y.
\]

\begin{definition}\label{def:hyperbolic}
    For $\alpha \geq 0$, the $\alpha$-Grushin hyperbolic plane $\bbH_\alpha$ is the sub-Riemannian structure on $\bbH^2$ induced from the vector field $X$ and $Y^\alpha$ given by
\[
X:=\partial_{x},~~Y^\alpha:=\frac{\abs{\sinh(x)}^{\alpha}}{\cosh(x)}\partial_{y}.
\]
The $\alpha$-Grushin hyperbolic half-plane $\overline{\bbH}^+_\alpha$ (resp. open half-plane $\bbH^+_\alpha$) is the subset of $\bbH_\alpha$ defined by
\[
\overline{\bbH}^+_\alpha := \{ p \in \bbH^2 \mid x \geq 0 \} \text{ (resp. $\bbH^+_\alpha := \{ p \in \bbH^2 \mid x > 0 \}$)}.
\]
\end{definition}
\noindent The $0$-Grushin hyperbolic plane is simply the two-dimensional Riemannian hyperbolic plane $\bbH^2$. When $\alpha > 0$, the $\alpha$-Grushin hyperbolic plane is a two-dimensional almost-Riemannian structure with $\{x = 0\}$ being its set of singular points.
% {\blue
To the best of our knowledge, this definition, although very natural, is new.
% } \todo{Sam: what do you think? Is that true? Ken: At least Pan and Honda didn't know, and it a.s. means that nobody knows.}.
At non-singular points, this sub-Riemannian structure admits the Riemannian metric
\begin{equation}
    \label{eq:HalphaRiemannianMetric}
    g_{\bbH_{\alpha}} = \diff x \otimes \diff x + \frac{\cosh^2(x)}{\sinh^{2\alpha}(x)} \diff y \otimes \diff y.
\end{equation}
A simple computation shows that the Riemannian volume induced from \cref{eq:HalphaRiemannianMetric} is given by
\[
\vol_{\bbH_\alpha}=\cosh(x)\abs{\sinh(x)}^{-\alpha}\diff x\diff y.
\]
We introduce the following weighted measure.
\begin{definition}\label{def:hyperbolicmeasure}
For $\beta\geq \alpha$,
we consider the weighted measure given by
\begin{equation*}
    \label{eq:MeasurealphaGrushinHyperbolic}
    \m^\beta_{\bbH_\alpha}:=\abs{\sinh(x)}^\beta\vol_{\bbH_\alpha}=\cosh(x)\abs{\sinh(x)}^{\beta-\alpha}\diff x \diff y=e^{-V_{\bbH_\alpha}}\vol_{\bbH_\alpha},
\end{equation*}
where $V_{\bbH_\alpha}(x,y):=-\beta\log\abs{\sinh(x)}$.
\end{definition}

\noindent As for the previous section, the $\alpha$-Grushin hyperbolic half-plane is a geodesically convex subset of $\bbH_\alpha$ and is a
% {\red uniquely}
geodesic space when seen as a length subspace of $\bbH_\alpha$. The open half-plane $\bbH_\alpha^+$ is also a geodesically convex subset and it is an incomplete Riemannian manifold since it doesn't contain any singular points of $\bbH_\alpha$.

\begin{prop}
    \label{prop:geodesicspacealphaGrushinHyperbolic}
     There is a
     % {\red unique}
     minimising geodesic contained within $\overline{\bbH}^+_\alpha$ that joins any two given points in $\overline{\bbH}^+_\alpha$. Furthermore, The $\alpha$-Grushin hyperbolic open half-plane $\bbH_\alpha^+$ is a geodesically convex subset of $\overline{\bbH}_\alpha^+$ and has the structure of an incomplete weighted
     % \todo{26. typo. DONE.}
     Riemannian manifold when equipped with the restriction of the measure $\mathfrak{m}^\beta_{\bbH_\alpha}$.
    % {\red In addition,
    % $\gamma$ is written by the image of the exponential map defined in \ref{expo}. Sam: Maybe we do not really use this information anymore?\\
    % Ken: Agree, maybe not. What we need is (strongly) geodesically convexity of $M$, and it can be proved without exponential map.}
\end{prop}

\begin{proof}
% {\red Sam: We add the proof or we explain how it is the same as the previous section?\\
% Ken: I think just mentioning the prvious proposition is enough}
Here, the Hamiltonian and the corresponding Hamilton's equation are given in Lobachevsky's coordinates, by
\begin{equation*}
    \label{eq:alphaGrushinHyperbolicHamilton}
    H(\lambda)=\frac{1}{2}\left[u^2+\frac{\sinh^{2\alpha}(x)}{\cosh^2(x)}v^2\right], \text{ and } \begin{cases}
    \dot{x}=u,\\
    \displaystyle \dot{y}=\frac{\sinh^{2\alpha}(x)}{\cosh^2(x)}v,\\
    \displaystyle \dot{u}= - 2 v^2 \ \sinh^{2(\alpha - 1)}(x) \tanh(x) (\alpha - \tanh^2(x)),\\
    \dot{v}=0.
\end{cases}
\end{equation*}
The rest of the proof follows exactly the arguments of \cref{prop:geodesicspacealphaGrushinSphere}.
\end{proof}

\subsection{The \texorpdfstring{$\infty$}{infty}-Grushin plane and half-plane}\label{section:alphaGrushin3}

The geometry of the so-called $\alpha$-Grushin plane, where $\alpha \geq 0$, has been studied in \cite{li2012}, \cite{borza2022}, and \cite{borza2023}. The $\alpha$-Grushin half-plane and the validity of the $\cd$ condition in this space is studied in \cite{rizzi2023failure}, and we recalled some details in \cref{section:Introduction} . Instead, we introduce a model of a Grushin plane with infinite Hausdorff dimension. The global chart $(x, y)$ simply denotes the cartesian coordinates in this section.

\begin{definition}\label{def:infty}
    The $\infty$-Grushin plane $\mathbb{G}_\infty$ is the sub-Riemannian structure on $\R^2$ induced from the vector field $X$ and $Y$ given by
\[
X:=\partial_{x},~~Y:= e^{-1/\abs{x}}\partial_{y}.
\]
The $\infty$-Grushin half-plane $\overline{\mathbb{G}}^+_\infty$ (resp. open half-plane $\mathbb{G}^+_\infty$) is the subset of $\mathbb{G}_\infty$ defined by
\[
\overline{\mathbb{G}}^+_\infty := \{ p \in \R^2 \mid x \geq 0 \} \text{ (resp. $\mathbb{G}^+_\infty := \{ p \in \R^2 \mid x > 0 \}$)}.
\]
\end{definition}

\begin{lemma}
\label{lemma:infinitehausdorffdimension}
    The $\infty$-Grushin plane and half-plane have infinite Hausdorff dimension.
\end{lemma}
% \todo{27. There should be a short proof: since, locally, $\di_\infty\geq \di_\alpha$ for any $\alpha$, $\dim \G_\infty\geq \dim \G_\alpha$ for any $\alpha$. Yes, let's change\\ Ken: Changed. DONE.}
\begin{proof}
We will show that the Hausdorff dimension of $\mathcal{S}:=\{(0,y)\mid y\in\R\}\subseteq \G_\infty$ is $+\infty$. Let us denote by $\di_\alpha$ (resp. $\di_\infty$) the induced distance on $\G_\alpha$ (resp. $\G_\infty$).
It is well-known that the Hausdorff dimension of $\Scal\subseteq \G_\alpha$ is $\alpha+1$, see e.g. \cite{Franchi1983}. Since the inequality $|x|^\alpha\geq e^{-1/|x|}$ holds for sufficiently small $|x|$,
we have the inequality $\di_\alpha\geq \di_\infty$ in a small neighbourhood of an arbitrary point in $\Scal$. This implies that $\dim_H(\Scal,\di_\alpha)\leq \dim_H(\Scal,\di_\infty)$ and concludes the lemma.
\end{proof}

The $\infty$-Grushin plane is a two-dimensional almost-Riemannian structure with two-dimensional almost-Riemannian structure with $\{x = 0\}$ being its set of singular points.
% {\blue
To the best of our knowledge, this definition of $\infty$-Grushin plane is also new.
%} \todo{Sam: what do you think? Is that true?K: I think so. If somebody pointed out, let's add it.}.
At non-singular points, this sub-Riemannian structure admits the Riemannian metric
\begin{equation}
    \label{eq:GinftyRiemannianMetric}
    g_{\mathbb{G}_{\infty}} = \diff x \otimes \diff x + e^{2/\abs{x}} \diff y \otimes \diff y.
\end{equation}
A simple computation shows that the Riemannian volume induced from \cref{eq:GinftyRiemannianMetric} is given by
\[
\vol_{\mathbb{G}_\infty}=e^{1/\abs{x}}\diff x\diff y.
\]
We introduce the following weighted measure.
\begin{definition}\label{def:infty_measure}
For
% {\red $\beta, \gamma \geq 0$}{\blue
$\beta\geq 0$ and $\gamma>0$
% }
,
we consider the weighted
% {\blue
(Radon)
% }
measure given by
\begin{equation*}
    \label{eq:MeasurealphaGrushininfty}
    \m^{\beta, \gamma}_{\mathbb{G}_{\infty}}:=\abs{x}^\beta e^{-\gamma/x^2}\vol_{\mathbb{G}_{\infty}}=\abs{x}^\beta e^{-\gamma/x^2 + 1/\abs{x}}\diff x\diff y=e^{-V}\vol_{\mathbb{G}_{\infty}},
\end{equation*}
where $V_{\G_\infty}(x,y):=\frac{\gamma}{x^2}-\beta\log\abs{x}$.
\end{definition}
% \todo[inline]{Sam: We should explain why this is a well-defined Radon measure. Is that the case for every $\beta, \gamma \geq 0$?\\
% Ken: The following argument is correct? If $\beta\geq 0$ and $\gamma>0$, then the measure does not diverges (locally finite). Other continuity property holds since the density function is continuous.\\
% Sam: Should this be made rigourous in a remark?}

The next result is analogous to the corresponding one in the previous two sections. The $\infty$-Grushin half-plane is a geodesically convex subset of $\mathbb{G}_\infty$ and a geodesic space when seen as a length subspace of $\mathbb{G}_\infty$. Similarly, the open half-plane $\mathbb{G}_\infty^+$ is also a geodesically convex subset and it is an incomplete Riemannian manifold since it doesn't contain any singular points of $\mathbb{G}_\infty$.

\begin{prop}
     There is a minimising geodesic contained within $\overline{\mathbb{G}}^+_\infty$ that joins any two given points in $\overline{\mathbb{G}}^+_\infty$. Furthermore, the $\infty$-Grushin open half-plane $\mathbb{G}_\infty^+$ is a geodesically convex subset of $\overline{\mathbb{G}}_\infty^+$ and has the structure of an incomplete weighted Riemannian manifold when equipped with the restriction of the measure $\mathfrak{m}^{\beta, \gamma}_{\mathbb{G}_\infty}$.
    % {\red In addition,
    % $\gamma$ is written by the image of the exponential map defined in \ref{expo}. Sam: Maybe we do not really use this information anymore?\\
    % Ken: Agree, maybe not. What we need is (strongly) geodesically convexity of $M$, and it can be proved without exponential map.}
\end{prop}

\begin{proof}
The argument is again analogous to the proofs of \cref{prop:geodesicspacealphaGrushinSphere} and \cref{prop:geodesicspacealphaGrushinHyperbolic}, with the Hamiltonian and Hamilton's equation given in cartesian coordinates by
\begin{equation*}
    \label{eq:alphaGrushinInfHamilton}
    H(\lambda)=\frac{1}{2}\left[u+e^{-2/\abs{x}}v^2\right], \text{ and } \begin{cases}
    \dot{x}=u,\\
    \displaystyle \dot{y}=e^{-2/\abs{x}}v,\\
    \displaystyle \dot{u}= - v^2 \frac{e^{-2/\abs{x}}}{\abs{x}^2} x,\\
    \dot{v}=0.
\end{cases}
\end{equation*}
\end{proof}

Hereafter, we will collectively refer to the $\alpha$-Grushin plane (resp. half-plane), the $\infty$-Grushin plane (resp. half-plane), the $\alpha$-Grushin sphere (resp. hemisphere), and the $\alpha$-Grushin hyperbolic plane (resp. half-plane) as the $\alpha$-Grushin spaces (resp. half-spaces).

\section{Equivalence between \texorpdfstring{$\cd(K,N)$}{CD(K,N)} and \texorpdfstring{$\Ric_N\geq K$}{RicN>=K} for some almost Riemannian manifolds}
\label{section:equivalence}

The $\alpha$-Grushin spaces and half-spaces introduced in \cref{section:themodels} are metric measure spaces $(\X, \di, \mathfrak{m})$ with a weighted Riemannian manifold $(M, \diff_g, e^{-V}\vol_g)$ as their interior. For $N \in (-\infty,0)\cup[n,+\infty]$
% \todo{28. Here the domain should be $(-\infty,0)\cup (n,+\infty)$, with the convention that if $N=n$ then $V\equiv $ const... ?}
, we recall that the $N$-Ricci tensor of an $n$-dimensional weighted Riemannian $(M, \diff_g, e^{-V} \vol_g)$ is defined by
% \todo{29. DONE $\Ric_N\to \Ric_{N,V}$}
\begin{equation}
    \label{eq:NRicci}
    \Ric_{N,V}:= \begin{cases}
        \Ric & \text{ if } N=n, \\
        \Ric + \hes(V) & \text{ if } N=+\infty,\\
        \Ric+\hes(V)-\dfrac{\diff V\otimes\diff V}{N-n} & \text{ otherwise},
    \end{cases}
\end{equation}
with the convention that $V$ must be constant when $N=n$.

The following theorem provides sufficient conditions under which the differential condition $\Ric_{N,V} \geq K$ on $M$ is equivalent to $(\X, \di, \mathfrak{m})$ satisfying the $\mathsf{CD}(K, N)$ condition. The proof generalises the sketch found in \cite[Section 3.5]{rizzi2023failure}, which is specific to the $\alpha$-Grushin half-plane. This theorem is, to some extent, related to the conjecture stated in \cite{han2020}.
\begin{theorem}
    \label{theorem:equivalence}
    Let $K \in \R$, $N \in \ointerval{-\infty}{0} \cup \interval{n}{+\infty}$, $(\X, \di, \mathfrak{m})$ be a metric measure space and $M$ be an open subset of $\X$ such that
	\begin{enumerate}[label=\normalfont(\roman*)]
	\item $M$ is a geodesically convex
    % \todo{30. what is geodesically convex. DONE. }
    subset of $(\X, \di)$, i.e. for every $x,y\in M$, there is a geodesic joining $x$ and $y$ and any such curve is contained in $M$,
	\item $(M, \di|_M, \mathfrak{m}|_M)$ possesses a weighted $n$-dimensional Riemannian manifold structure,
	\item $\mathfrak{m}(\X \setminus M) = 0$.
 %{\red and $\X\setminus M\subset \partial M$}\todo{Sam: I still think we can easily remove it. See in the proof}.
	\end{enumerate}
	Then, the metric measure space $(\X, \di, \mathfrak{m})$ satisfies the $\cd(K,N)$ condition if and only if $\mathrm{Ric}_{N,V} \geq K$ on $M$.
\end{theorem}
% \todo[inline]{Ken: I think Remark 1.9 is not good. Since, to prove their $BE$-condition, we need to check Sobolev-to-Lipschitz property, see \cite{honda2018}.
% I suggest to prove the equivalence the $Ric^N\geq K$ and $\cd$, and use \cite{ledonne-lucic-pasqualetto2023}. Check the arguments in the end of this section.\\ Sam: Yes I agree. We might combine it in the main theorem? or separate like your wrote?\\ Sam again: so you mean that in general it might not be true, we are only sure it's true for $\alpha$-Grushin models.\\
% Ken: Yes, I guess $RCD$ holds in the general setting, however, it requires further discussions.}
\begin{proof}
   We provide the proof only for $N \in \interval{n}{+\infty}$. The details for $N \in \ointerval{-\infty}{0}$ are exactly the same, but one needs to replace the relevant key results used in the proof by the analogous ones when $N$ is negative. In particular, the Gromov--Hausdorff convergence which we use below must be replaced by the pointed $\mathsf{iKRW}$ convergence described in \cite{chiara-mattia-sosa2023}.

   Firstly, we note that the $\cd(K,N)$ condition on $(\X, \di, \mathsf{m})$ directly implies that $\mathrm{Ric}_{N,V} \geq K$ on $M$ by \cite[Theorem 17.36]{villani} (see also \cite[Part (e) of the proof of Theorem 1.7]{sturm2006}). We therefore focus now on the other implication.

	For a closed metric subspace $A\subseteq \X$, we shall denote by $W_2^A$ the Wasserstein distance on $A$. Consider $\mu_0, \mu_1 \in \mathscr{P}_2^{\mathrm{ac}}(\X, \di, \mathfrak{m})$ with continuous densities $\rho_0$ and $\rho_1$ respectively. We need to argue that there exists a $W_2^\X$-geodesic $(\mu_s)_{s \in \interval{0}{1}}$ joining $\mu_0$ to $\mu_1$ with $\mu_s = (e_s)_\sharp\nu$ for some $\nu \in \mathscr{P}(\mathrm{Geo}(\X, \di))$ such that the inequality in \cref{def.2.1} is satisfied.
    % the $N'$-Rényi entropy is $K$-convex
    % \todo{31. refer the relevant inequality in $\cd(K,N)$? DONE.}
    % along $(\mu_s)_{s \in \interval{0}{1}}$ for all $N' \geq N$.
    For all $k \in \N$ and given an arbitrary $x_0 \in \X$, we define the sets
\[
M_k:=\mathrm{cl}(\mathrm{geo}(N_k)), \text{ and } \ \ N_k := \{x\in M\mid k \geq \di(x_0,x), \text{ and }\di(x,\X\setminus M)\geq 1/k\},
\]
where $\mathrm{geo}(A)$ is the geodesic hull of a subset $A \subseteq \X$, i.e. the union of all geodesics starting at $x \in A$ and ending at $y \in A$. The set $N_k$ is clearly closed and bounded. Since $(\X, \diff, \mathfrak{m})$ is a complete metric measure space, the Heine--Borel property (see \cite[Theorem 2.5.28]{burago-burago-sergei2001}) implies that $N_k$ is also compact in $\X$.
% that $\mathrm{geo}(\{x\in M\mid k \geq \di(x,\X\setminus M)\geq 1/k\})$ is also bounded. Since $(\X, \diff, \mathfrak{m})$ is a (complete) metric measure space, .

    \paragraph{Step 1. The closure of the geodesic hull of a compact subset of $M$ is in $M$.}

    Let $K$ be a compact set of $\X$ contained in $M$. We claim that $\mathrm{cl}(\mathrm{geo}(K))$ is contained within $M$. Since $M$ is assumed to be geodesically convex in $\X$, we note that their closure in $\X$ satisfies $\mathrm{cl}(\mathrm{geo}(K)) \subseteq \mathrm{cl}(M)$. Because $M$ is open in $\X$, we also have that $\diff(x, \partial M) \geq \diff(K, \partial M) > 0$ for all $x \in K$. Suppose by contradiction that there exists $p \in \mathrm{cl}(\mathrm{geo}(K))$ such that $p \in \partial M$. This means that for all $n \in \N$, there exist $x_n, y_n \in K$, a geodesic $\gamma_n$ joining $x_n$ to $y_n$ contained in $M$, and a point $p_n \in M$ lying on $\gamma_n$ such that $p_n \to p$ as $n \to +\infty$. By compactness, we have, after extracting a converging subsequence, that $x_n \to x$ and $y_n \to y$ for some $x, y \in K$. In particular, the geodesics $\gamma_n$ have uniformly bounded lengths. Since $(\X, \diff)$ is a complete locally compact length space, the Arzela-Ascoli theorem from \cite[Theorem 2.5.14]{burago-burago-sergei2001} implies that the sequence of curves $\gamma_n$ contains a uniformly converging subsequence. This limit, which we denote by $\gamma$, is a curve of $\X$ with endpoints $x$ and $y$ and passing through $p$. Actually, the curve $\gamma$ is a minimising geodesics between $x$ and $y$ by \cite[Proposition 2.5.17]{burago-burago-sergei2001}. By geodesic convexity of $M$ assumed at (ii), the curve $\gamma$ is contained in $M$ while the point $p \notin M$, leading to a contradiction.

    % \paragraph{Step 2. The support of $\mu_0$ and $\mu_1$ are compact and contained in $M$.}
    % Let $U:=\overline{\mathrm{Cvxhul}}(\supp(\mu_0)\cup\supp(\mu_1))$.
    % By the assumption (ii) and (iii), there is a unique optimal transport plan $\nu\in\Prob(\Geo(\X,\di))$ joining $\mu_0$ and $\mu_1$.
    % Note that $\supp(\nu)\subset U$ holds by the assumption (i). Furthermore, the results in \cite{Fathi-Figalli2010} imply that there exists a unique $W_2$-geodesic $(\mu_s)_{s \in \interval{0}{1}}$ joining $\mu_0$ to $\mu_1$, induced from the transportation map $T_s(x)=\exp_x(s\nabla \phi(x))$ for some $\mu_0$-almost everywhere differentiable function $\phi$. Combined with the assumptions $\mathrm{dim}(M) \leq N$ and $\mathrm{Ric}_{V}^N \geq K$ on $ U\subset M$, \cite[Thm\, 17.36]{villani} implies that the $N$-Rényi entropy is $K$-convex along $(\mu_s)_{s \in \interval{0}{1}}$.

    \paragraph{Step 2. The support of $\mu_0$ and $\mu_1$ are compact and contained in $N_k$.}
    Assume that $\supp(\mu_0)$ and $\supp(\mu_1)$ are compact and that $\supp(\mu_0) \cup \supp(\mu_1) \subseteq N_k$ for some $k \in \N$. The previous step applied to the compact $N_k$ implies that each pair of points $x \in \supp(\mu_0)$ and $y \in \supp(\mu_1)$ are connected by a geodesic contained in $M$. By \cite[Theorem 13]{mccann}, there exists a unique $W_2^{M_k}$-geodesic $(\mu_s)_{s \in \interval{0}{1}}$ joining $\mu_0$ to $\mu_1$. By \cite[Proposition 2.10]{lott--villani2009}
    % \todo{Cor. 7.22? Sam: Maybe just chapter 7? Ken: The optimal transference plan $\Pi$ in the book is your $\nu?$\\Sam: I think yes, but it's better to check}
    , we also know that there exists $\nu \in \mathscr{P}(\mathrm{Geo}(\X, \di))$ such that $\mu_s = (e_s)_\sharp \nu$ for all $s \in \interval{0}{1}$ and $\mathrm{supp}(\nu) \subseteq \Gamma := (e_0 \times e_1)^{-1}(\mathrm{supp}(\mu_0) \times \mathrm{supp}(\mu_1))$. A geodesic in $\gamma \in \Gamma$ joins points in $\mathrm{supp}(\mu_0)$ with points in $\mathrm{supp}(\mu_1)$ and is contained in $M$ thanks to Step 1. Since $\mathrm{dim}(M) \leq N$ and $\mathrm{Ric}_{V}^N \geq K$ on $(M, g)$, \cite[Theorem\, 17.36]{villani} (or also \cite[Parts (a), (b), (c), and (d) of the proof of Theorem 1.7]{sturm2006}) implies that the inequality in \cref{def.2.1} for all $N' \geq N$.
    % $N'$-Rényi entropy is $K$-convex along $(\mu_s)_{s \in \interval{0}{1}}$ for all $N' \geq N$.
% \todo[inline]{ken: If we do not use pmGH convergence, then we may use \cite[Theorem 5.20, 29.20]{villani}, and arguments in Corollary 29.23}
    \paragraph{Step 3. The support of $\mu_0$ and $\mu_1$ are compact and contained in $\X$.}
    We can assume without loss of generality that $\X \setminus M \subseteq \partial M$. Indeed, when it comes to the $\mathsf{CD}(K, N)$ condition or the Gromov--Hausdorff convergence that we are going to discuss in this step, only the support of the measure matters (see \cite[Remark 3.1]{gigli2015}).
    % \todo{Sam: Remard 3.1 in Gigli for GH and find another reference for CD}
    The definition of $N_k$ and assumption (iii) imply that $N_k \to \X$ as $k \to +\infty$ in the pointed measured Gromov--Hausdorff convergence (see \cite[Definition 3.9]{gigli2015}), when taking the inclusion map $\iota_k:N_k\to \X$ as approximation maps.
 %    {\blue From this construction, there are sequences of probability measures $(\mu_0^k),(\mu_1^k)$ with compact in $N_k$ which converges to $\mu_0,\mu_1$ in $W_2^\X$.
 %    This is argued in
 % \cite[Proposition 28.7]{villani} where it is shown that the inclusion map $(\iota_k)_\#$ gives an approximation map $(\prob(N_k),W_2^{N_k})\to (\prob(\X),W_2^\X)$
 % % ,under the assumption that the support of $\mu_0^k,\mu_1^k,\mu_0,\mu_1$ are in a compact subset $K\subset \X$, which we can assume without loss of generality
 % }
Although $N_k$ is not a geodesically convex subspace, we have shown in Step 1 that $M_k := \mathrm{cl}(\mathrm{geo}(N_k)) \subseteq M$.
Then, an argument similar to the proof of \cite[Theorem 28.13]{villani} shows that the metric space $(\prob(N_k),W_2^{N_k})$ converges to $(\prob(\X),W_2^\X)$ in the geodesic local Gromov--Hausdorff topology, via the inclusion map $(\iota_k)_\#$. Hence there are sequences of probability measures $(\mu_i^k)_{k\in\N}$ $(i=0,1)$, supported on $N_k$, such that their pushforwards $(\iota_k)_\#\mu_i^k$ converges to $\mu_i$ in $(\prob_2(\X),W_2^\X)$.
    % \todo[inline]{Sam: why? I don't see why $\X \setminus M \subseteq \partial M$ is important.\\
    % Ken: If we don't have it, then it admits the example Half-Grushin $\cup\{(x,0)\mid x\leq 0\}$. Maybe we can ignore $\{(x,0)\mid x\leq 0\}$ by considering noly on the support of reference measure. I wanted to skip this arguments.}
    %{\red Note that for any $k\in\N$, there is $k'\in \N$ such that $\mathrm{geo}(N_k)\subseteq N_{k'}$. By the proof of \cite[Theorem 28.13]{villani}, the $(\iota_{k'})_{\#}|_{\prob(M_k)}$ is also an approximation map and thus there are probability measures $\mu_i^k$ $(i=0,1)$ such that the supports are contained in $N_k$ and that $(\iota_{k'})_\#\mu_i^k \to\mu_i$ in $W_2$-distance.}\todo{Sam: I am still unsure about this step, let's discuss it again.}
    % \todo[inline]{Sam: We need to introduce $\mu_0^k$ and $\mu_1^k$ and argue why they exist.}
    Let $(\mu_s^k)_{s \in \interval{0}{1}}$ be the Wasserstein geodesic joining $\mu_0^k$ and $\mu_1^k$ and $\nu_k$ be the optimal transport plan joining $\mu_0^k$ to $\mu_1^k$ in $M_k$. By Step 3, the inequality in \cref{def.2.1} holds
    % $N'$-Rényi entropy is $K$-convex
    along $(\mu_s^k)_{s \in \interval{0}{1}}$ for all $N' \geq N$. Furthermore, by \cite[Theorem 28.9 and Exercise 28.15]{villani}, $\nu_k$ weakly converges to an optimal transport plan $\nu$ in $\X$ joining $\mu_0$ to $\mu_1$, and the $W_2^{M_k}$-geodesics $(\mu_s^k)_{s \in \interval{0}{1}}$ uniformly converges to a $W_2^{\X}$-geodesic $(\mu_s)_{s \in \interval{0}{1}}$ in $\X$ joining $\mu_0$ to $\mu_1$. By the same argument as in \cite[Theorem\, 29.24 and 29.21]{villani}, we conclude that the inequality in \cref{def.2.1} remains valid
    % $N'$-Rényi entropy is also $K$-convex
    along $(\mu_s)_{s \in \interval{0}{1}}$ for all $N' \geq N$, by lower semicontinuity.

    \paragraph{Step 4. The general case.} The case where the support of $\mu_0$ and $\mu_1$ are not necessarily compact is obtained by the previous step and by using an exhaustion by compact sets, see \cite[Appendix E]{lott--villani2009}.
    % \todo{Sam: find a right reference, maybe \cite[Theorem 29.20(i,iii)]{villani} {\red find better reference}.}
\end{proof}

As for the $\mathsf{RCD}(K, N)$ condition (for $K \in \R$ and $N \in \interval{1}{+\infty}$), it is enough for our purposes to use the result of \cite[Theorem 1.2]{ledonne-lucic-pasqualetto2023}, which shows that
sub-Riemannian manifolds equipped with a non-negative Radon measure are infinitesimally Hilbertian. Although the results in \cite{ledonne-lucic-pasqualetto2023} assume the bracket-generating condition and our $\alpha$-Grushin half-spaces do not when $\alpha \notin \N$, their Finsler approximation techniques can still be applied.
It is sufficient to regard the $\alpha$-Grushin half-spaces as $\alpha$-Grushin (full) spaces equipped with a measure supported on the half-space, for the following theorem to follow directly.
\begin{theorem}
% \todo[inline]{32. Instead of using LLP23, can we use the stability of infinitesimal Hilbertianity within $\cd$?\\
% Ken: Probably not. Since our $M_k$ may not me geodesically convex, and in particular, it may not be geodesic $\cd(K,N)$ space. To prove the stability, we need the $1$-st order convergence i.e. the convergence of the gradient flow in $L^2$ or the null $(1,2)$-capacity condition, which are unclear to us. Therefore, we can not say that $M_k$ is $\mathsf{CD}/\mathsf{RCD}$.}
     For $K \in \R$ and $N \in \interval{1}{+\infty}$, the $\cd(K,N)$ condition is equivalent to the $\RCD(K,N)$ condition for the $\alpha$-Grushin half-spaces.
\end{theorem}

\section{Generalised Ricci curvature of \texorpdfstring{$\alpha$}{alpha}-Grushin half-spaces}
\label{section:riccicomputation}

The previous section shows that establishing the validity of the $\mathsf{CD}(K, N)$ condition in the $\alpha$-Grushin half-spaces introduced in \cref{section:themodels} is equivalent to a computation of the Bakry--\'Emery Ricci curvature. The following simple computation will be handy.

\begin{lemma}
\label{e:Rictensor}
    Let $(M, g)$ be a $2$-dimensional Riemannian manifold, $\mathcal{U} \subseteq M$ an open set, and $N \in (-\infty, 0) \cup (2, +\infty]$. Assume that $(x, y) : \mathcal{U} \to \mathbb{R}^2$ is a chart and that we are given two smooth functions $f : \mathcal{U} \to \mathbb{R}\setminus \{0\}$ and $V : \mathcal{U} \to \mathbb{R}$ that only depends on $x$. If, in this coordinate,
    \[
    g = \diff x \otimes \diff x + \frac{1}{f(x)^2} \diff y \otimes \diff y,
    \]
    then it holds that
\begin{equation*}
        \begin{aligned}
            \Ric_{N,V} = \left[ \left(\frac{f'}{f}\right)' - \left(\frac{f'}{f}\right)^2 + V'' - \frac{(V')^2}{N - 2}\right] \diff x \otimes \diff x
            + \frac{1}{f^2}\left[\left(\frac{f'}{f}\right)' - \left(\frac{f'}{f}\right)^2 - \frac{f'}{f} V'\right] \diff y \otimes \diff y,
        \end{aligned}
\end{equation*}
where all the derivatives are understood with respect to the variable $x$.
% \begin{equation}\label{e:Rictensor}
%     \begin{aligned}
%        \Ric_{N,V}=
%        \left[\frac{\partial_x^2f\cdot f-2(\partial_xf)^2}{f^2}+\partial_x^2 V - \frac{(\partial_x V)^2}{N-2}\right]dx\otimes dx\\
%        \qquad\qquad+ \left[\frac{\partial_xf\cdot f-2(\partial_x f)^2}{f^4}-\frac{\partial_x f}{f^3}\partial_x V\right]dy\otimes dy.
%     \end{aligned}
% \end{equation}
\end{lemma}

\begin{proof}
The two vector fields $X = \partial_x$ and $Y = f(x) \partial_y$ defined on $\mathcal{U}$ form a family of $g$-orthonormal fields. The only non-zero bracket relation is
\[
[X, Y] = \frac{f'}{f} Y.
\]
Using the Koszul formula, we easily obtain
\[
\nabla_X X = \nabla_X Y = 0, \qquad \nabla_Y X = - \frac{f'}{f} Y, \qquad \nabla_Y Y = \frac{f'}{f} X.
\]
It follows that the only non-zero entry of the Riemann curvature tensor is
\[
\mathrm{R}(X, Y, X, Y) = \left(\frac{f'}{f}\right)' - \left(\frac{f'}{f}\right)^2.
\]
Using the symmetries, we obtain
\[
\mathrm{Ric} = \left[\left(\frac{f'}{f}\right)' - \left(\frac{f'}{f}\right)^2\right] \diff x \otimes \diff x + \frac{1}{f^2}\left[\left(\frac{f'}{f}\right)' - \left(\frac{f'}{f}\right)^2\right] \diff y \otimes \diff y.
\]
We also find that $\nabla V =X(V)X$, and that the only non-zero entries of the Hessian are
\[\hes(V)(X,X) = V'',~~\hes(V)(Y,Y)=-\frac{f'}{f}V'.\]
Finally, one has the non-zero entry
\[\diff V\otimes \diff V(X,X)=(V')^2,\]
and the proof is complete with \cref{eq:NRicci}.
\end{proof}

\paragraph{The $\alpha$-Grushin hemisphere.}
We make use of the coordinate chart and the notation described in \cref{section:alphaGrushin}.

% {\red \begin{prop}
% \label{prop:NRiccialphaGrushin}
% The $N$-Ricci curvature of the $\alpha$-Grushin open hemisphere $(\Sbb_\alpha^+, \diff_{\S_\alpha}, \mathfrak{m}^\beta_{\S_\alpha})$ is given by
% \begin{align*}
%     \Ric_{N,V}=&-\frac{1}{\sin^2(x)}\left[3\alpha-1+(\alpha-1)^2\cos^2(x) - \beta + \frac{\beta^2}{N-2}\cos^2(x) \right]\diff x\otimes \diff x \\
%     &-\frac{\cos^2(x)}{\sin^{2\alpha+2}(x)}\left[3\alpha-1+(\alpha-1)^2\cos^2(x) - \beta \left(1+(\alpha-1)\cos^2(x)\right)\right]\diff y\otimes \diff y.
% \end{align*}
% \end{prop}

% \begin{proof}
%     Using \cref{eq:christoffelsymbols}, we get that the Christoffel symbols of this Riemannian space are given by
% \begin{align*}
% &\Gamma_{ij}^1=\begin{cases}
%     0 & (i,j)=(1,1),(1,2),(2,1),\\
%     \displaystyle \frac{\cos(x)\left[1+(\alpha-1)\cos^2(x)\right]}{\sin^{2\alpha}(x)\sin(x)} & (i,j)=(2,2),\end{cases}\\
% &\Gamma_{ij}^2=\begin{cases}
%     0 & (i,j)=(1,1),(2,2),\\
%     \displaystyle -\frac{1+(\alpha-1)\cos^2(x)}{\sin(x)\cos(x)} & (i,j)=(1,2),(2,1).
% \end{cases}
% \end{align*}
% The expression for $\Ric_N$ follows from \cref{eq:pureRicci}, \cref{eq:Hessian}, and \cref{eq:NRicci} together with $V(x, y) =-\beta\log\abs{\sin(x)}$ from \cref{def:MeasurealphaGrushinSphere}.
% \end{proof}

% A necessary and sufficient condition to have a lower bound on $\Ric_N$ can now be found.
% }
\begin{prop}
    Given $K \in \R$, $N \in \ointerval{-\infty}{0} \cup \linterval{2}{+\infty}$, $\alpha \geq 0$ and $\beta \geq \alpha$, the $N$-Ricci curvature of the $\alpha$-Grushin open hemisphere $(\Sbb_\alpha^+, \diff_{\S_\alpha}, \mathfrak{m}^\beta_{\S_\alpha})$ is
\begin{align*}
    \Ric_{N,V}=&-\frac{1}{\sin^2(x)}\left[3\alpha-1+(\alpha-1)^2\cos^2(x) - \beta + \frac{\beta^2}{N-2}\cos^2(x) \right]\diff x\otimes \diff x \\
    &-\frac{\cos^2(x)}{\sin^{2\alpha+2}(x)}\left[3\alpha-1+(\alpha-1)^2\cos^2(x) - \beta \left(1+(\alpha-1)\cos^2(x)\right)\right]\diff y\otimes \diff y.
\end{align*}
Furthermore, it holds $\Ric_{N,V} \geq K$ if and only if
    \[
    \beta-\alpha^2-\alpha+\min\left(-K+(\alpha-1)^2,~-\frac{\beta^2}{N-2},~\beta(\alpha-1)\right)\geq 0.
    \]
\end{prop}

\begin{proof}
 The first part is done by \cref{e:Rictensor} with $f(x)=\frac{\abs{\sin(x)}^\alpha}{\cos(x)}$ and $V(x)=-\beta\log\abs{\sin(x)}$. For the second part, we see by comparing the coefficients of the tensors $\diff x\otimes \diff x$ and $\diff y\otimes \diff y$ respectively that $\Ric_{N,V}\geq K$ holds if and only if
\begin{equation*}\label{ineq:Ricbound1}
-3\alpha+1+\beta-K+\left[K-(\alpha-1)^2-\frac{\beta^2}{N-2}\right]\cos^2(x)\geq 0\end{equation*}
and
\begin{equation*}\label{ineq:Ricbound2}
    -3\alpha+1+\beta-K+\left[K-(\alpha-1)^2+\beta(\alpha-1)\right]\cos^2(x)\geq 0.
\end{equation*}
Since $\cos(x)$ takes its value in $[0,1)$ on $\S_\alpha^+$,
the above two inequalities are equivalent to the following inequality:
\[-3\alpha+1+\beta-K+\min\left(0,K-(\alpha-1)^2-\frac{\beta^2}{N-2},K-(\alpha-1)^2+\beta(\alpha-1)\right)\geq 0,\]
which is trivially equivalent to the inequality in the statement.
\end{proof}

\paragraph{The $\alpha$-Grushin hyperbolic half-plane.}

% The unweighted Riemannian (Popp's) measure is
% \[\vol=\cosh(x)|\sinh(x)|^{-\alpha}\diff x\diff y.\]
% For $\beta\geq 0$,
% we consider the weighted measure
% \[\m_\beta:=|\sinh(x)|^\beta\vol=\cosh(x)|\sinh(x)|^{\beta-\alpha}\diff x\diff y=e^{-f}\vol,\]
% where $f(x,y):=-\beta\log|\sinh(x)|$.

We make use of the coordinate chart and the notation described in \cref{section:alphaGrushin2}.

% {\red \begin{prop}
% \label{prop:NRiccialphaGrushin2}
% The $N$-Ricci curvature of the $\alpha$-Grushin open half hyperbolic plane $(\bbH_\alpha^+, \diff_{\bbH_\alpha}, \mathfrak{m}^\beta_{\bbH_\alpha})$ is given by
% \begin{align*}
%     \Ric_{N,V}=&-\frac{1}{\sinh^2(x)}\left[3\alpha-1+(\alpha-1)^2\cosh^2(x) - \beta + \frac{\beta^2}{N-2}\cosh^2(x) \right]\diff x\otimes \diff x \\
%     &-\frac{\cosh^2(x)}{\sin^{2\alpha+2}(x)}\left[3\alpha-1+(\alpha-1)^2\cosh^2(x) - \beta \left(1+(\alpha-1)\cosh^2(x)\right)\right]\diff y\otimes \diff y.
% \end{align*}
% \end{prop}

% \begin{proof}
%     Using \cref{eq:christoffelsymbols}, we get that the Christoffel symbols of this Riemannian space are given by
% \begin{align*}
% &\Gamma_{ij}^1=\begin{cases}
%     0 & (i,j)=(1,1),(1,2),(2,1),\\
%     \displaystyle \frac{\cosh(x)\left[1+(\alpha-1)\cosh^2(x)\right]}{\sinh^{2\alpha}(x)\sinh(x)} & (i,j)=(2,2),\end{cases}\\
% &\Gamma_{ij}^2=\begin{cases}
%     0 & (i,j)=(1,1),(2,2),\\
%     \displaystyle -\frac{1+(\alpha-1)\cosh^2(x)}{\sinh(x)\cosh(x)} & (i,j)=(1,2),(2,1).
% \end{cases}
% \end{align*}
% The expression for $\Ric_N$ follows from \cref{eq:pureRicci}, \cref{eq:Hessian}, and \cref{eq:NRicci} together with $V(x, y) =-\beta\log\abs{\sinh(x)}$ from \cref{def:hyperbolicmeasure}.
% \end{proof}

% A necessary and sufficient condition to have a lower bound on $\Ric_N$ can now be found.}

\begin{prop}
    Given $K \in \R$, $N \in \ointerval{-\infty}{0} \cup \linterval{2}{+\infty}$, $\alpha \geq 0$ and $\beta \geq \alpha$, the $N$-Ricci curvature of the $\alpha$-Grushin open hyperbolic half-plane $(\bbH_\alpha^+, \diff_{\bbH_\alpha}, \mathfrak{m}^\beta_{\bbH_\alpha})$ is
\begin{align*}
    \Ric_{N,V}=&-\frac{1}{\sinh^2(x)}\left[3\alpha-1+(\alpha-1)^2\cosh^2(x) - \beta + \frac{\beta^2}{N-2}\cosh^2(x) \right]\diff x\otimes \diff x \\
    &-\frac{\cosh^2(x)}{\sin^{2\alpha+2}(x)}\left[3\alpha-1+(\alpha-1)^2\cosh^2(x) - \beta \left(1+(\alpha-1)\cosh^2(x)\right)\right]\diff y\otimes \diff y.
\end{align*}
    Furthermore, it holds $\Ric_{N,V} \geq K$ if and only if
    \[
    \min\left(-K-(\alpha-1)^2,\beta-\alpha^2-\alpha\right)+\min\left(-\frac{\beta^2}{N-2},~\beta(\alpha-1)\right)\geq 0
    \]
\end{prop}

\begin{proof}
 The first part is done by \cref{e:Rictensor} with $f(x)=\frac{\abs{\sinh(x)}^\alpha}{\cosh(x)}$ and $V(x)=-\beta\log\abs{\sinh(x)}$. For the second part, we see by comparing the coefficients of the tensors $\diff x\otimes \diff x$ and $\diff y\otimes \diff y$ respectively that $\Ric_{N,V}\geq K$ holds if and only if
\begin{equation*}\label{ineq:Ricbound1-2}
-3\alpha+1+\beta+K+\left[-K-(\alpha-1)^2-\frac{\beta^2}{N-2}\right]\cosh^2(x)\geq 0\end{equation*}
and
\begin{equation*}\label{ineq:Ricbound2-2}
    -3\alpha+1+\beta+K+\left[-K-(\alpha-1)^2+\beta(\alpha-1)\right]\cosh^2(x)\geq 0.
\end{equation*}
Since $\cosh(x)\in[1,+\infty)$,
the two inequalities hold for all $x\geq 0$ if and only if
 \[\min\left(0,-3\alpha+1+\beta+K\right)+\min\left(-K-(\alpha-1)^2-\frac{\beta^2}{N-2},-K-(\alpha-1)^2+\beta(\alpha-1)\right)\geq 0\]
holds, which is equivalent to the inequality in the statement.

\end{proof}

\paragraph{The $\infty$-Grushin half-plane.}

% The unweighted Riemannian (Popp's) measure is
% \[\vol=e^{\frac{1}{|x|}}\diff x\diff y.\]
% For $\beta,\gamma\geq 0$,
% we consider the weighted measure
% \[\m_{\beta,\gamma}:=|x|^\beta e^{-\frac{\gamma}{|x|^2}}\vol=|x|^\beta e^{-\frac{\gamma}{|x|^2}+\frac{1}{|x|}}\diff x\diff y=e^{-V}\vol,\]
% where $V(x,y):=\frac{\gamma}{|x|^2}-\beta\log|x|$,

We make use of the coordinate chart and the notation described in \cref{section:alphaGrushin3}.
% {\red
% \begin{prop}
% \label{prop:NRiccialphaGrushin3}
% The $N$-Ricci curvature of the $\infty$-Grushin open half-plane $(\G_\infty^+, \diff_{\G_\infty}, \mathfrak{m}^{\beta,\gamma}_{\G_\infty})$ is given by
% \begin{align*}
%     \Ric_{N,V}=&\left[\frac{6\gamma-1}{|x|^4}-\frac{2}{|x|^3}+\frac{\beta}{|x|^2}-\frac{1}{N-2}\left(\frac{4\gamma^2}{|x|^6}+\frac{4\beta\gamma}{|x|^4}+\frac{\beta^2}{|x|^2}\right)\right]\diff x\otimes \diff x\\
%     &+\left[\frac{2\gamma}{|x|^5}+\frac{1}{|x|^4}+\frac{\beta}{|x|^3}\right]e^{\frac{2}{|x|}}\diff y\otimes\diff y.
% \end{align*}
% \end{prop}

% \begin{proof}
%     Using \cref{eq:christoffelsymbols}, we get that the Christoffel symbols of this Riemannian space are given by
%     \begin{align*}
%     &\Gamma_{ij}^1=\begin{cases}
%     0 & (i,j)=(1,1),(1,2),(2,1),\\
%     \frac{e^{\frac{2}{|x|}}}{x|x|} & (i,j)=(2,2),\end{cases}\\
%     &\Gamma_{ij}^2=\begin{cases}
%     0 & (i,j)=(1,1),(2,2),\\
%     -\frac{1}{x|x|} & (i,j)=(1,2),(2,1).
% \end{cases}
% \end{align*}
% The expression for $\Ric_N$ follows from \cref{eq:pureRicci}, \cref{eq:Hessian}, and \cref{eq:NRicci} together with $V(x, y) =\frac{\gamma}{|x|^2}-\beta\log\abs{x}$ from \cref{def:infty_measure}.
% \end{proof}

% A necessary and sufficient condition to have a lower bound on $\Ric_N$ can now be found.
% }

\begin{prop}
    Given $K \in \R$, $N \in \ointerval{-\infty}{0} \cup \linterval{2}{+\infty}$ and $\beta,
    \gamma\geq 0$, the $N$-Ricci curvature of the $\infty$-Grushin open half-plane $(\G_\infty^+, \diff_{\G_\infty}, \mathfrak{m}^{\beta,\gamma}_{\G_\infty})$ is
\begin{align*}
    \Ric_{N,V}=&\left[\frac{6\gamma-1}{|x|^4}-\frac{2}{|x|^3}+\frac{\beta}{|x|^2}-\frac{1}{(N-2)x^6}\left(2\gamma+\beta x^2\right)^2\right]\diff x\otimes \diff x\\
    &+\left[\frac{2\gamma}{|x|^5}-\frac{1}{|x|^4}+\frac{\beta-2}{|x|^3}\right]e^{\frac{2}{|x|}}\diff y\otimes\diff y.
\end{align*}
    It holds $\Ric_{N,V} \geq K$ if and only if the following inequalities hold for any $x\geq 0$:
% {\red \[-\frac{\gamma^2}{N-2}+\left(-1+6\gamma-\frac{4\beta\gamma}{N-2}\right)t^2-2t^3+\left(\beta-\frac{\beta^2}{N-2}\right)t^4- Kt^6\geq 0,\]
% and
% \[2\gamma+t+\beta t^2-Kt^5\geq 0.\]}
\[(6\gamma-1)x^2-2x^3+\beta x^4-\frac{1}{N-2}(2\gamma+\beta x^2)^2\geq Kx^6~~\text{and}~~2\gamma-x+(\beta-2)x^2\geq Kx^5.\]
In particular, $\Ric_{\infty,V}\geq 0$ holds if and only if
\[\min\{\beta(6\gamma-1),8(\beta-2)\gamma\}\geq 1.\]
\end{prop}

\begin{proof}
% From \cref{prop:NRiccialphaGrushin3}, we see by comparing the coefficients of the tensors $\diff x\otimes \diff x$ and $\diff y\otimes \diff y$ respectively that $\Ric_N\geq K$ holds if and only if
% \begin{equation*}\label{ineq:Ricbound1-3}
% -3\alpha+1+\beta-K+\left[K-(\alpha-1)^2-\frac{\beta^2}{N-2}\right]\cosh^2(x)\geq 0\end{equation*}
% and
% \begin{equation*}\label{ineq:Ricbound2-3}
%     -3\alpha+1+\beta-K+\left[K-(\alpha-1)^2+\beta(\alpha-1)\right]\cosh^2(x)\geq 0.
% \end{equation*}
The proof is reached by using \cref{e:Rictensor} with $f(x)=e^{-1/|x|}$ and $V(x)=\frac{\gamma}{|x|^2}-\beta\log|x|$, and by comparing the coefficients of $\diff x\otimes \diff x$ and $\diff y\otimes \diff y$ respectively.
\end{proof}

\cref{thm:main1}, \cref{thm:main2}, \cref{thm:main3} follow from the Ricci computations this section, noting that the validity of the assumptions in \cref{theorem:equivalence} are verified because of the results of \cref{section:themodels}.

  \printbibliography
%\end{spacing}

% \bibliographystyle{alpha}
% \bibliography{bibliography}
\end{document}